\documentclass[12
pt]{amsproc}%
\RequirePackage{srcltx}
\usepackage{amsmath}
\usepackage{amssymb}
\usepackage{amsthm}
\def\leq {\leqslant}\def\le {\leqslant}
\def\ge {\geqslant}
\def\geq {\geqslant}

\usepackage{amssymb}
\usepackage{amsfonts}
\usepackage{amsmath}
\usepackage{graphicx}%
\providecommand{\U}[1]{\protect\rule{.1in}{.1in}}
\theoremstyle{plain}

\newtheorem{theorem}{Theorem}[section]
\newtheorem{lemma}[theorem]{Lemma}
\newtheorem{definition}[theorem]{Definition}
\newtheorem{corollary}[theorem]{Corollary}
\newtheorem{proposition}[theorem]{Proposition}
\theoremstyle{definition}
\newtheorem{remark}[theorem]{Remark}
\newtheorem{example}[theorem]{Example}
\numberwithin{equation}{section}
\def\N{{{\Bbb N}}}

\def\R{{\Bbb R}}
\def\M{\mathcal M}
\def\Mpl{\mathcal M\sp+}

\def\bx{\overline X}

\begin{document}

\title[Inequalities for K-functionals
] {
A unified  approach to  inequalities for K-functionals and  moduli of smoothness
}

\author{Amiran Gogatishvili}
\address{A. Gogatishvili \\
Mathematical Institute of the Czech Academy of Sciences\\
Zitn{\'a} 25, 115 67 Prague 1, Czech Republic}

\email{gogatish@math.cas.cz}

\author{Bohum{\'i}r Opic}


\address{Bohum\'{\i}r Opic,
Department of Mathematical Analysis,
Faculty of Mathematics and Physics, Charles University,
 Sokolovsk\'a 83, 186 75 Prague 8, Czech Republic}
\email{opic@karlin.mff.cuni.cz
\vskip0,08cm
\noindent
and}

\address {\vskip-0,5cm \noindent
Department of Mathematics, Faculty of Science, J. E. Purkyn\v e University, \v Cesk\' e ml\' ade\v ze~8, 400 96 \' Ust{\'\i} nad Labem, Czech Republic
}

\author{Sergey Tikhonov}

\address{S.~Tikhonov,
ICREA, Pg.
Llu\'is Companys 23, 08010 Barcelona, Spain\\
Centre de Recerca Matem\`atica,
Campus de Bellaterra, Edifici~C 08193 Bellaterra (Barcelona), Spain;  and Universitat Aut\`onoma de
Barcelona, Edifici~C 08193 Bellaterra (Barcelona), Spain.}

\email{stikhonov@crm.cat}

\author{Walter Trebels}
\address{W. Trebels \\
Department of Mathematics, AG Algebra,
Technical University of Darmstadt,
64289 Darmstadt, Germany}
\email{trebels@mathematik.tu-darmstadt.de}

\thanks{The research has been partially supported by
 by the grant P201-18-00580S of the Grant Agency of the Czech Republic,
PID2020-114948GB-I00,  2021 SGR 00087, 
  the CERCA Programme of the Generalitat de Catalunya,
 the  Ministry of Education and Science of the Republic of Kazakhstan AP14870758,
  and  by the Spanish State Research Agency, through the Severo Ochoa and Mar\'ia de Maeztu Program for Centers and Units of Excellence in R\&D (CEX2020-001084-M).
The research of Amiran  Gogatishvili was partially supported by  the grant project 23-04720S of the Czech Science Foundation (GA\v{C}R),  The Institute of Mathematics, CAS is supported  by RVO:67985840  and by  Shota Rustaveli National Science Foundation (SRNSF), grant no: FR21-12353.}

\begin{abstract}
The paper provides a detailed study of crucial inequalities for smoothness and interpolation characteristics in rearrangement invariant  Banach function spaces.
We present a unified  approach based on Holmstedt formulas to obtain these estimates. As examples,
we derive new inequalities for moduli of smoothness and $K$-functionals in various Lorentz spaces.
\end{abstract}
\maketitle

\vskip 0.5cm

\section{Introduction}
Some, nowadays well-known, inequalities between  moduli of continuity,
or more general, between moduli of smoothness are attached to the
names of Marchaud,  Ul'yanov, and Kolyada.
These inequalities play an
important role in approximation theory as well as in the theory of
function spaces, in particular, they
can be used to derive embedding properties of
function
spaces with fixed degree of smoothness, see, e.g.,
  \cite[Section~5.4]{besh}, \cite{besov1}, \cite{devore}.

The purpose of this  paper is to consider crucial
inequalities (Marchaud, Ul'yanov, etc.)
from an abstract point of view. To this end, in Section~4 we assume suitable embeddings between interpolation and potential spaces
(the interpolation spaces may be interpreted as abstract Besov spaces).
Simultaneously, abstract versions of the Holmstedt formulas
are developed, which allow also  to cover
limiting cases. In Section~5  applications are given in the case of
general  weighted Lorentz spaces. Finally, Section~6 deals with
applications to Lorentz-Karamata spaces.

  To illustrate our results, we start in Subsection~1.1 with the
formulation of the aforesaid basic
inequalities adapted to Lebesgue spaces $\, L_p(\mathbb{R}^n), \, 1<p<\infty.$
Their improvements and extensions in the framework of Lorentz spaces $\, L_{p,r}(\mathbb{R}^n)$
 (note that $\, L_{p,p}= L_p$) are described in Subsection~1.2, proofs are given in Section~2.

\subsection{Some basic results}
 A detailed study of inequalities between different moduli of
smoothness on
$L_p(\mathbb{R}^n),$ $1\le p \le \infty$,  can be naturally divided into
two parts: inequalities for moduli of smoothness of different orders in
$\, L_p$  and inequalities in different metrics $\, (L_p,L_{p*})$.
In the paper a modulus of smoothness of order $\, \kappa >0$ on an r.i.
function space $\, X$ (defined in Section 3, e.g., $\, X=L_p$)
is given by
\begin{equation} \label{fracdif}
 \omega_\kappa(f,t)_X=\sup_{|h|\le t }  \left\|
\Delta_{h}^{\kappa} f(x) \right\|_{X}, \quad \mbox{where }\ \
\Delta_{h}^{\kappa} f(x) =\sum\limits_{\nu=0}^\infty(-1)^{\nu}
\binom{\kappa}{\nu} f\,\big(x+\nu h\big).
\end{equation}

Let us begin with the key inequalities on $\, L_p(\mathbb R^n)$.
Trivially, if  $k, m, n \in {\mathbb N}$ and $1\le p\le\infty,$ then
\begin{equation}\label{trivest}
\omega_{k+m}(f,t)_{L_p} \lesssim \omega_k(f,t)_{L_{p}}
\quad\mbox{for all } \ t>0 \mbox{ and } f \in L_p( \mathbb R^n).
\end{equation}
In 1927 Marchaud \cite{mar} proved  his
famous  inequality (being a weak inverse of (\ref{trivest})):
Given $k, m, n \in {\mathbb N}$ and $1\le p\le\infty,$ then
\begin{equation}\label{mar1}
\omega_{k}(f,t)_{L_{{p}}} \lesssim t^k \int_{t}^\infty u^{-k }
\omega_{k+m}(f,u)_{L_{{p}}} \frac{du}{u}
\quad\mbox{for all } \ t>0 \mbox{ and } f \in L_p( \mathbb R^n).
\end{equation}
Using geometric properties of the $\, L_p$ spaces
when $\, 1<p<\infty$, in 1958  M. F. Timan
improved (\ref{mar1}) (see, e.g., \cite[Chapter~2, Theorem~8.4]{devore}): If $k, m, n \in {\mathbb N}$, $1<p<\infty,$
 and $q=\min\{2,p\},$
 then 
 \begin{equation}\label{mar2}
 \omega_k(f,t)_{L_{p}}  \lesssim
 t^k \left( \int_{t}^\infty \Big[
u^{-k } \omega_{k+m}(f,u)_{L_{{p}}}\Big]^q
\frac{du}{u}\right)^{1/q} \quad
\mbox{for all } \ t>0 \mbox{ and } f \in L_p( \mathbb R^n).
\end{equation}
Observe the natural formal passage from (\ref{mar2}) to (\ref{mar1}) when
$\, p \to 1+.$\\
In 2008 F. Dai, Z. Ditzian and S. Tikhonov \cite{ddt} derived an
improvement of (\ref{trivest}): If $\,k, m, n \in {\mathbb
N},$   $\, 1<p<\infty,$ and $ r=\max \{2, p\},$  then 
\begin{equation}\label{ddt1}
t^k
\left( \int_{t}^\infty \Big[ u^{-k }
\omega_{k+m}(f,u)_{L_{{p}}}\Big]^r
\frac{du}{u}\right)^{1/r} \lesssim \omega_k(f,t)_{L_{p}} \quad
\mbox{for all } \ t>0 \mbox{ and } f \in L_p( \mathbb R^n).
\end{equation}
Observe again the natural formal passage
from (\ref{ddt1}) to (\ref{trivest}), this time when $\, p \to \infty.$
We call (\ref{ddt1}) a reverse Marchaud inequality (in \cite{ddt} it is
called a sharp Jackson inequality).

\smallskip
Consider now inequalities for moduli of smoothness in different
Lebesgue metrics.
In 1968 P.L. Ul'yanov \cite{ul1} proved  such an inequality
  for  periodic functions in $\,
L_p(\mathbb T).$ Its $\, \mathbb R^n$-counterpart reads as
follows (see, e.g., \cite{bom}):
If $k, n\in {\mathbb N},$ $1 \le p<\infty$,
$0<\delta < \min\{n/p, k\}$,
  and $1/p^*=1/p-\delta/n,$ then
  \begin{equation}\label{ulul}
\omega_k(f,t)_{L_{p^*}}\lesssim
\left( \int_{0}^t \Big[ u^{-\delta}
\omega_{k}(f,u)_{L_{{p}}}\Big]^{p^*} \frac{du}{u}\right)^{1/p^*}
\quad  \mbox{as } \ \,t\to 0+
\end{equation}
holds 
for all $f \in L_p(\mathbb R^n)$ (for which the right-hand side of\eqref{ulul} is finite).
\footnotemark\footnotetext{\ One can show that if $f \in L_p(\mathbb R^n)$  and the right-hand side of \eqref{ulul} is finite
for some $t>0,$ then $f \in L_{p^*}(\mathbb R^n)$  and so the modulus of smoothness appearing on the left-hand side
of \eqref{ulul} is well defined.  Note that we always look at inequalities involving moduli of smoothness
in different metrics at this way. One can also show that if $f \in L_p(\mathbb R^n)$ and the right-hand side of \eqref{ulul} is finite for some $t>0,$ then it is  finite for all $t>0$ - cf. Remark \ref{remark---r} mentioned below).
}

In 1988  V.I.~Kolyada \cite{kolyada} gave a definite strengthening of
(\ref{ulul}) on $\, L_p(\mathbb T^n).$  In the $\, \mathbb R^n$-setting
his result is the following (see \cite{goldman}): 

\noindent
Suppose that  $k,n\in \mathbb{N},$ and either $1<p<\infty$ and $\, n \ge 1$, or $\, p=1$ and $\, n\ge
2.$ If 
$0<\delta < \min\{n/p, k\}$ and $ 1/p^* = 1/p -\delta/n,$ then, for all $f\in L_p(\mathbb R^n)$, 
\begin{equation}\label{kol1}
\hskip-0,5cm
t^{k-\delta} \left( \int_{t}^{\infty}\![u^{\delta-k}\omega_k(f,u)_{L_{p^*}}]^p
\frac{du}{u}
\right)^{1/p}\! \lesssim \left( \int_{0}^{t}[u^{-\delta}\omega_k(f,u)_{L_{p}}]^{p^*}
\frac{du}{u} \right)^{1/{p^*}}\! \ \   \mbox{as } \ t\to 0+.
\end{equation}

Another   extension
 of (\ref{ulul}), which is
 not comparable with inequality (\ref{kol1}), is the so-called sharp Ul'yanov inequality proved in 2010 
independently in
 \cite{simtik} and \cite{tre4}:
\\
If  $k,n \in {\mathbb N},$ $1<p<\infty,$  $0<\delta < n/p$,  and $1/p^*=1/p-\delta/n,$ then, for all $f\in L_p(\mathbb R^n),$
\begin{equation}\label{kol12}
\omega_k(f,t)_{{L_{p^*}}}
\lesssim \left( \int_{0}^{t}[u^{-\delta}\omega_{k+\delta}(f,u)_{L_{p}}]^{p^*}
\frac{du}{u} \right)^{1/{p^*}} \ \   \mbox{as } \ t\to 0+.
\end{equation}

In the case $p=1$ (\ref{kol12}) does not hold in general \cite[Theorem~1(B)]{tikhonov} and it requires some modifications \cite[Rem.~6.20]{paper with oscar} (see also \cite[Theorem~1(A)]{tikhonov}).
If  $k,n \in {\mathbb N},$   $0<\delta < n$,  and $1/p^*=1-\delta/n,$ then, for all $f\in L_1(\mathbb R^n),$
\begin{equation*}
\omega_k(f,t)_{{L_{p^*}}}
\lesssim \left(
\int_{0}^
{
t
(|\ln t|)^{1/(k p^*)}
}[u^{-\delta}\omega_{k+\delta}(f,u)_{L_{1}}]^{p^*}
\frac{du}{u} \right)^{1/{p^*}} \ \   \mbox{as } \ t\to 0+.
\end{equation*}

The importance of these inequalities instigated much research
in various areas of analysis (theory of function spaces, approximation theory,
interpolation theory)
 and led to
 numerous publications. We mention only a few recent papers:
  \cite{
   dom, dom++, dom+++, oscar, gorb, hatr, jum,  kolo, kolyada1, pez, tikhonov, tre4}.
  Basic properties of moduli of smoothness of functions from $L_p(\mathbb R^n),$ $0 < p \le \infty$, are given in \cite{kolo1}.


\subsection{Inequalities for moduli of smoothness on Lorentz spaces}
We say that a~measurable function $f$ belongs to  the Lorentz space
$\, L_{p,r}=L_{p,r} ({\mathbb R}^n),\, 1\le p,r\le \infty,\, $
if  (see, e.g., \cite[Section~4.4]{besh})
\[
\| f\|_{p,r}:= \left\{ \begin{array}{l@{\; \; ,\quad}l}
 \Big(\int_{0}^{\infty}[t^{1/p}\,
f^*(t)]^r \frac{dt}{t}\Big) ^{1/r}<\infty  & r<\infty,\\
\quad \sup\limits_{t>0}t^{1/p}\, f^*(t)<\infty  & r=\infty ,\\
\end{array} \right.
\]
where $\, f^*$ denotes the non-increasing rearrangement of $\, f.$
Thus $\, L_p =L_{p,p}$ and $\, \| f \|_p =\| f\|_{p,p}.$

The next  statements extend the inequalities mentioned above to Lorentz spaces.

\smallskip
{\sc Proposition 1.1.} {\it If $\,n \in {\mathbb N},\, 1 < p < \infty,\;  1 \le q_0, q_1,
r_0,r_1 \le \infty ,\; r_0 \le r_1,$ and  $\, \beta >0,$ then, for all 
 $f \in L_{p,r_0}({\mathbb R}^n)$}:\\[1mm]
{\rm (A)} {\sf Marchaud-type inequality.}
\begin{equation}\label{march}
\omega_\beta(f,t)_{L_{p,r_1}}\lesssim  t^\beta
\left(  \int_{t}^\infty \Big[ u^{-\beta }
\omega_{\beta+\sigma}(f,u)_{L_{p,r_0}}\Big]^{q_0}
\frac{du}{u}\right)^{1/q_0}  \ \   \mbox{as } \ t\to 0+
\end{equation}
{\it provided $\, \sigma>0$ and $\, q_0 \le \min \{ p,2,r_1\}$ if $\, p
\not=2.$ If $\, p=2$ and $ \, r_0 \le 2$, then take $\, q_0 \le \min \{
2,r_1\},$ and in the case $\, p=2,\; r_0 > 2$ one has to take
$\, q_0 <2.$
}
\\[1mm]
\medskip
{\rm (B)} {\sf Reverse Marchaud-type inequality.}
\begin{equation}\label{revmarch}
t^\beta  \left(  \int_{t}^\infty \Big[ u^{-\gamma}
\omega_{\beta+\gamma}(f,u)_{L_{p,r_1}}\Big]^{q_1}
\frac{du}{u}\right)^{1/q_1} \lesssim \omega_\beta(f,t)_{L_{p,r_0}} \ \   \mbox{as } \ t\to 0+
\end{equation}
({\it with usual modification if $q_1=\infty)$
 provided $\, \gamma>0$ and $\, q_1 \ge \max \{ p,2,r_0\}$ if $\,
p \not=2.$ If $\, p=2$ and $ \, r_1 \ge 2$, then take $\, q_1 \ge \max \{
2,r_0\},$ and in the case $\, p=2,\; r_1 < 2$ one has to take
$\, q_1 >2.$
}

\bigskip
Denote by $\, W_p^k({\mathbb R}^n), \;
1 \le p <\infty,\; k \in {\mathbb N},$ the Sobolev space of order
$\, k,$ i.e., $\, f\in~W_p^k({\mathbb R}^n)$ if $\, f$ and all its
(weak)  derivatives up to the order $\, k$ belong to $\,
L_p({\mathbb R}^n).$
It is well known that, by Taylor's formula,
\[
\omega_{m+k}(f,t)_{L_p} \lesssim t^k \sum_{|\mu|=k}^{}\omega_m (D^{\,\mu}
f,t)_{L_p} \, , \quad m \in {\mathbb N},\; \mu \in {\mathbb N}_0^n ,
\mbox{ for all}\, f \in W_p^k({\mathbb R}^n) \ \mbox{and} \ t>0.
\]
Here we use the multi-index notation $\, |\mu|:=
\sum_{j=1}^{n}\mu_j ,\;  D^{\,\mu} = \prod_{j=1}^{n}\, (\partial
/\partial x_j)^{\mu_j}\, .$
We want to state an improvement and some type of reverse of this inequality in the case $\,
1<p<\infty.$ To this end, we need Besov spaces and Riesz potential spaces, both
modelled upon Lorentz spaces. 

\bigskip
We make use of  the
Fourier analytical approach in $\, {\mathcal S}'$ (cf. \cite{belo}):\\
Take a $ C^\infty$-function $\varphi$ such that
\begin{equation}\label{SmoothFunction}
        \text{supp } \varphi \subset \{x \in \mathbb{R}^n : |x| \leq
7/4\} \quad \text{ and }\;  \varphi(x) = 1 \text{ if } |x| \leq 3/2.
\end{equation}
For $j \in \mathbb{Z}$ and $x \in \mathbb{R}^n$, let
\begin{equation}\label{resolution}
        \varphi_j(x) = \varphi(2^{-j}x) - \varphi(2^{-j+1}x).
\end{equation}
The sequence $\{\varphi_j\}_{j \in \mathbb{Z}}$ is a smooth
dyadic resolution of unity, i.e.,  $ 1=\sum_{j=-\infty}^\infty
\varphi_j (x) $ for all $ x \in \mathbb{R}^n,\, x\not= 0$.

\medskip
Let $1 \leq p,q,r \leq \infty$ and $\sigma>0$. The Besov space
$B^{\sigma}_{(p,r),q}(\mathbb{R}^n)$ consists of all $f\in~L_{p,r}
({\mathbb R}^n)$ such that
\begin{equation}\label{besov1}
|f |_{B^{\,\sigma}_{(p,r),q}} = \Bigg(\sum_{j=-\infty}^\infty
\Big[2^{j \sigma} \| {\mathcal F}^{-1}[\varphi_j]*f
\|_{L_{p,r}}\Big]^q\Bigg)^{1/q} < \infty
\end{equation}
(the sum should be replaced by the supremum if $q=\infty$).
Here the symbol $\, {\mathcal F}^{-1}$ is used for the inverse Fourier
transform.   An equivalent
characterization of this semi-norm in terms of moduli of smoothness is
given  by
\begin{equation}\label{besov2}
    |f |^{*}_{B^{\,\sigma}_{(p,r),q}} = \left( \int_{0}^{\infty}
\Big[t^{-\sigma} \omega_k(f,t)_{L_{p,r}}\Big]^q\frac{dt}{t}
\right)^{1/q} ,\qquad 0<\sigma<k.
\end{equation}

\medskip
The  Riesz potential space $\, H^{\sigma}_{p,r}({\mathbb R}^n),\, \sigma
\ge 0,$ consists of all  $f \in L_{p,r}({\mathbb R}^n)$ for which
\begin{equation}\label{riesz}
|f|_{H^{\sigma}_{p,r}} := \|D_R^{\,\sigma} f\|_{L_{p,r}} <\infty \, ,\,\,\,
\mbox{where }\; \, D_R^{\,\sigma} f:= \sum_{j=-\infty}^{\infty} {\mathcal
F}^{-1} [|\xi|^\sigma\varphi_j]*f
\end{equation}
(the $\,\sigma$-th Riesz derivative)
converges in $\, {\mathcal S}'$ to an $\, L_{p,r}({\mathbb
R}^n)$-function. Note that $\, W_p^k=H_{p,p}^k$ if $\, 1<p<\infty.$

\smallskip

{\sc Proposition 1.2.} {\it Let $\, n \in {\mathbb N}, \; 1 < p < \infty,\;
 1 \le q_0, q_1 \le \infty ,\; 1 \le r_0 =r_1=r\le \infty,$ and $ \, \beta,\sigma >0.$\\[2mm]
{\rm (A)} \;
 If $\, f \in L_{p,r}({\mathbb R}^n)\,$ then,
under the assumptions
on the parameters $\, q_0$ and $r$ of  \, \mbox{{\rm Proposition}}
\mbox{{\rm  1.1 (A)}},  for all $\,t>0$,
\begin{equation}\label{deriv1}
\omega_\sigma (D^\beta_R f,t)_{L_{p,r}} \lesssim \left( \int_{0}^{t}
\Big[u^{-\beta} \omega_{\beta+\sigma}(f,u)_{L_{p,r}}\Big]^{q_0}
\frac{du}{u}\right)^{1/q_0}. 
\end{equation}
In particular, if $\, \beta = m$ and $\sigma=k \in {\mathbb N},$ then,
for all $\, \mu\in {\mathbb N}_0^n$ with $\, |\mu|= m$,
\begin{equation}\label{deriv2}
\omega_k (D^\mu f,t)_{L_{p,r}} \lesssim \left( \int_{0}^{t} \Big[u^{-m}
\omega_{k+m} (f,u)_{L_{p,r}}\Big]^{q_0} \frac{du}{u} \right)^{1/q_0} .
\end{equation}


\medskip
{\rm (B)}  \; If $\, f \in H^\beta_{p,r}({\mathbb R}^n)\, $ then, under
the assumptions on the parameters $\, q_1$ and $r$ of  \,
\mbox{{\rm Proposition}}  \mbox{{\rm  1.1 (B)}},   for all $\,t>0$,
\begin{equation}\label{deriv1*}
 \left( \int_{0}^{t}
\Big[u^{-\beta} \omega_{\beta+\sigma}(f,u)_{L_{p,r}}\Big]^{q_1}
\frac{du}{u}\right)^{1/q_1}  \lesssim  \omega_\sigma (D^\beta_R
f,t)_{L_{p,r}}. 
\end{equation}
In particular, if $\, \beta = m$ and $\sigma=k \in {\mathbb N},$ then
\[
\left( \int_{0}^{t}
\Big[u^{-m} \omega_{m+k}(f,u)_{L_{p,r}}\Big]^{q_1}
\frac{du}{u}\right)^{1/q_1}  \lesssim  \sup _{j=1,\ldots ,n}\omega_k
\Big( \frac{\partial^m f}{\partial x_j^m},t \Big)_{L_{p,r}}. 
\]
}

Finally consider inequalities between moduli of smoothness in
different metrics.

\smallskip
{\sc Proposition 1.3.} {\it Suppose $n\in \mathbb{N},$ $1 < p < \infty,\; 0<\delta<n/p, \,1/p^*=1/p-\delta/n,
1 \le q_0,q_1,
r_0,r_1 \le \infty ,\;$ and $\beta >0.$ 
\\[1mm]
{\rm (A)} {\sf Sharp Ul'yanov inequality.} {\it If $r_0, q_1 \le r_1,$
then, for all $\, t>0$ and $f \in L_{p,r_0}(\mathbb{R}^n),$ }
\begin{equation}\label{ulya}
\omega_\beta (f,t)_{L_{p^*,r_1}} \lesssim \Big(
\int_{0}^{t}
[u^{-\delta}  \, \omega_{\beta +\delta}(f,u)_{L_{p,r_0}}]^{q_1}
\frac{du}{u}\Big)^{1/q_1}.
\end{equation}
\\[1mm]
{\rm (B)} {\sf Kolyada-type inequality.} {\it If $\, r_0 \le q_0,\; q_1 \le
r_1,$ then, for all $\, t>0$ and $f \in L_{p,r_0}(\mathbb{R}^n),$ }
\begin{equation}\label{kolya}
t^{\beta} \left( \int_{t}^{\infty} \! [\, u^{-\beta}
\omega_{\beta +\delta}(f,u)_{L_{p^*,r_1}}]^{q_0}
\frac{du}{u}\right)^{1/q_0} \!
\lesssim \left( \int_{0}^{t}\! [u^{-\delta}\omega_{\beta
+\delta}(f,u)_{L_{p,r_0}} ]^{q_1} \frac{du}{u} \right)^{1/q_1}\! \! \!
\! \! \!.
\end{equation}
}
\section{Remarks and proofs in outlines }
Peetre's $\, K$-functional $K_0$ for the compatible couple $\, (L_{p,r},
H^\sigma_{p,r})$ plays a decisive role in the proofs of Propositions 1.1--1.3.
It is defined by
\[
K_0(f,t;L_{p,r},H^\sigma_{p,r}) = \inf_{g\in H^\sigma_{p,r}} (\|
f-g\|_{p,r} +t |g|_{H^\sigma_{p,r}}), \qquad f\in L_{p,r},\quad t>0.
\]
We also need the characterization, for
$1<p<\infty, \; \sigma >0,\;  1\le r \le \infty,$
\begin{equation}\label{wilmes}
K_0(f,t^\sigma;L_{p,r},H^\sigma_{p,r}) \approx
\omega_\sigma(f,t)_{L_{p,r}},\quad f\in L_{p,r},\quad t>0,
\end{equation}
(see \cite{wil} and its extension in
\cite[(1.13)]{gott})
and the identification of the interpolation space 
 given by
\[
(L_{p,r} , H_{p,r}^\sigma)_{\theta,q}= B^{\,\theta \sigma}_{(p,r),q}\, ,\qquad
\sigma >0,\; 0<\theta<1,\;1<p<\infty,\, 1\le r,q \le \infty,
\]
where $\, (\cdot , \cdot )_{\theta,q}\,$ denotes Peetre's real
interpolation method. The improvements and extensions of inequalities
(\ref{mar1})--(\ref{kol12}) can
be easily proved via the Holmstedt formulas \cite[Section~5.2]{besh}.
One only needs to exchange  in \cite{tre4} the embeddings between Besov
and potential spaces modelled on Lebesgue
spaces by the corresponding ones modelled on Lorentz spaces.
Therefore, we only sketch the proofs of the propositions stated in Subsection~1.2.

\smallskip
Concerning (\ref{march}) and (\ref{revmarch}), note that, under the
restrictions on $\, q_0$ and $\, q_1$ given in Proposition 1.1, the following
embeddings  are true:
\begin{equation}\label{marchemb}
B^{\,\sigma}_{(p,r_0),q_0} \hookrightarrow H^{\sigma}_{p,r_1}
\end{equation}
 if
$1 \le p < \infty,$\;  $1 \le q_0, r_0,r_1 \le \infty,$\; $r_0 \le r_1$
(see Theorem 1.1,  (iv)--(vi) in \cite{see})\hspace{.05cm}
and
\begin{equation}\label{revmarchemb}
H^{\gamma}_{p,r_0} \hookrightarrow B^\gamma_{(p,r_1),q_1}
\end{equation}
{ if}
$ 1 \le p < \infty,$ \;  $1 \le q_1, r_0,r_1 \le \infty,$\; $r_0 \le r_1$
(see Theorem 1.2, (iv)--(vi) in \cite{see}).


\medskip
\begin{remark}\label{rem2.1}
In parts (i) and (ii) of this remark we assume the
same  restrictions on the parameters under which (\ref{march}) and
(\ref{revmarch}) hold,  respectively.

(i) Divide equation (\ref{march}) by $\, t^{-\beta}$
and let $\, t \to 0+.$ Then on the right-hand side one gets $\, |f|^*
_{B^{\,\beta}_{(p,r_0),q_0}}. $ One way how to handle the left-hand side is to
introduce the generalized Weierstra{\ss} means
$\, W_t^\beta f = {\mathcal F}^{-1}[e^{(t|\xi|)^\beta}] *f $. By
\cite[(1.11)]{gott},
 one has
\[
K_0(f,t^\beta; L_{p,r_1}, H^\beta_{p,r_1}) \approx \| f- W_t^\beta
f\|_{p,r_1},\, \quad f\in L_{p,r_1},\quad t>0.
\]
Also, by
\cite[Corollary~3.4.11]{berens},
\[
\lim _{t \to 0+} t^{-\beta} \| f - W_t^\beta f \|_{p,r_1} \approx |f|
_{H^{\beta}_{p,r_1}}\, .
\]
Hence, in view of (\ref{wilmes}), (\ref{march}) implies
(\ref{marchemb}).  In particular,
(\ref{march}) and (\ref{marchemb}) are equivalent assertions. This means the following:
if  inequality  (\ref{march}) holds under certain range of parameters, then embedding  (\ref{marchemb}) is valid
for such parameters and vice versa.

(ii) 
If (\ref{revmarch}) is true, then its right-hand side is
equivalent
to $\, K_0(f,t^\beta;L_{p,r_0}, H^\beta_{p,r_0})$, which trivially is smaller
than $\, t^\beta |f|_{H^\beta_{p,r_0}}.$ Dividing  inequality
(\ref{revmarch}) by $\, t^\beta$, one gets
\[
 \left(  \int_{t}^\infty \Big[
u^{-\gamma} \omega_{\beta+\gamma}(f,u)_{L_{p,r_1}}\Big]^{q_1}
\frac{du}{u}\right)^{1/q_1} \lesssim |f|_{H^\beta_{p,r_0}}
\]
uniformly in $\, t>0,$ and (\ref{revmarchemb}) follows. Thus,
(\ref{revmarch}) and (\ref{revmarchemb}) are again equivalent statements.
\hfill                                                 $\Box$
\end{remark}

\smallskip
Concerning Proposition 1.2 (A), let $\, f \in B^{\,\beta}_{(p,r),q_0}$. Then, by
(\ref{marchemb}), $\, f \in H^\beta_{p,r}\, ,$ hence $\, D^{\,\beta}_R f \in
L_{p,r} \, ,$ and
\begin{equation}\label{mirek1}
 \omega_\sigma(D^{\,\beta}_R f,t)_{L_{p,r}} \lesssim \| D^{\,\beta}_R
f -h\|_{L_{p,r}} + t^\sigma |h|_{H^\sigma_{p,r}}\quad\mbox{for all}\ h \in
H^\sigma_{p,r}.
\end{equation}
If $\, g \in H^{\sigma+\beta}_{p,r}$, then $\, D_R^{\,\beta}\, g\in
H^{\sigma}_{p,r},\; |D_R^{\,\beta} \,g |_{H^{\sigma}_{p,r}}
=|g|_{H^{\sigma + \beta}_{p,r}}\, $ and $\, \| D_R^{\,\beta}
(f-g)\|_{L_{p,r}} \lesssim |f-g|_{B^{\,\beta}_{(p,r),q_0}}\, .$ Now choose $\,
h=D^{\,\beta}_R\, g$ in (\ref{mirek1}) to obtain
\[
\omega_\sigma(D^{\,\beta}_R f,t)_{L_{p,r}} \lesssim
|f-g|_{B^{\,\beta}_{(p,r),q_0}} +t^\sigma |D_R^{\,\beta}\, g|_{H^{\sigma }_{p,r}}
\approx |f-g|_{(L_{p,r},H^{\beta+\sigma}_{p,r})_{\theta,q_0}} + t^\sigma
 |g|_{H^{\sigma + \beta}_{p,r}} ,
\]
where in Peetre's $\, (\cdot ,\cdot)_{\theta ,q_0}$-interpolation method
 one has to put \mbox{$\, \theta = \beta /(\beta+\sigma).$} Taking the
minimum  over all $\, g \in H^{\sigma + \beta}_{p,r}$ in the last
display and using the appropriate Holmstedt formula (\cite[p. 310]{besh}), we arrive at
(\ref{deriv1}).

\smallskip
Regarding (\ref{deriv2}), observe that the $\, j$-th Riesz transform $R_j$,
$\, 1 \le j \le n,$ (with the Fourier symbol $\, \xi_j /|\xi|,\;
\xi \in {\mathbb R}^n$)
 is a bounded operator from $\, L_p $ into \mbox{$\, L_p ,\;
1<p<\infty ,$} hence also bounded from $\, L_{p,r} $ into $\, L_{p,r} ,\;
1<p<\infty ,\; 1 \le r \le \infty .$ Now set $\, {\mathcal R}^\mu :=
\prod_{j=1}^{n} R_j^{\,\mu_j}$ to obtain $\, \| D^{\,\mu} f\|_{L_{p,r}} = \|
{\mathcal R}^\mu D^{\,|\mu|}_R f\|_{L_{p,r}} \lesssim  \| D^{\,|\mu|}_R
f\|_{L_{p,r}}\, .$ Hence,
\[
\omega_k(D^{\,\mu} f,t)_{L_{p,r}} = \sup _{|y|\le t} \| \Delta^k_y {\mathcal
R}^\mu D_R^{\,|\mu|} f\|_{L_{p,r}}= \sup _{|y|\le t} \| {\mathcal R}^\mu
\Delta^k_y D_R^{\,m} f\|_{L_{p,r}} \lesssim \omega_k(D_R^{\,m} f,t)_{L_{p,r}}
\]
and (\ref{deriv2}) follows from (\ref{deriv1}).

Concerning Proposition 1.2 (B), we follow the argument starting with (12.13)  in \cite{oscar}. Thus,
 by
 \cite[Lemma~1.4 with $\alpha=0$]{gott},
\[
\omega_\sigma (D^\beta_R f,t)_{L_{p,r}} \approx  K_0(D^\beta_R
f,t^\sigma; L_{p,r}, H^\sigma_{p,r}) \approx \| D^\beta_R(f-V_t
f)\|_{p,r} + t^\sigma |D^\beta_R V_t f|_{H^\sigma_{p,r}},\,
\]
where $\, V_tf$ are the de la
Vall{\'e}e-Poussin means of $\, f.$
Now use
 Theorem 1.2 (iv) - (vi) in \cite{see}, subsequently,
 the lifting property of Besov spaces, and again \cite[Lemma~1.4]{gott} to obtain
\begin{eqnarray*}
\omega_\sigma (D^\beta_R f,t)_{L_{p,r}}
&\gtrsim& | D^\beta_R(f-V_t f)|_{B^0_{(p,r),q_1}} + t^\sigma |D^\beta_R V_t
f|_{H^\sigma_{p,r}}
\\
&\approx& | f-V_t f|_{B^\beta_{(p,r),q_1}} +
t^\sigma |V_t f|_{H^{\beta+\sigma}_{p,r}}
\approx  K_0(f,t^\sigma; B^\beta_{(p,r),q_1},H^{\beta+\sigma}_{p,r}).
\end{eqnarray*}
Since $\, B^\beta_{(p,r),q_1} =
(L_{p,r},H_{p,r}^{\beta+\sigma})_{\theta,q_1}\, , \; \beta= \theta
(\beta+\sigma)$ (see, e.g.,  \cite[Theorem 6.3.1]{belo}), hence $\, 1-\theta
= \sigma/(\beta+\sigma)$ and, therefore, by the  Holmstedt
formula,  we finally derive
\[
\left( \int_{0}^{t}[ u^{-\beta} \omega_{\beta+\sigma}(f,u)_{L_{p,r}}
]^{q_1} \frac{du}{u} \right) ^{1/q_1} \lesssim
\omega_\sigma (D^\beta_R f,t)_{L_{p,r}}\, .
\]
In particular, if $\, \beta =m \in {\mathbb N},$ then, for even $\, m$
and hence $\, \gamma_j \in 2 {\mathbb N_0}$,
\[
 D^\beta_R f =
{\mathcal F}^{-1}\big[(\xi_1^2+ \cdots + \xi_n^2)^{m/2}{\mathcal F}[f]\big]
=  \sum_{|\gamma|=m}^{} {\mathcal F}^{-1}\big[\prod_{j=1}^{n}
\xi_j^{\gamma_j} {\mathcal F}[f]\big]\, ,\qquad \gamma \in {\mathbb N_0}^n.
\]
 If $\, \gamma_j$ is odd, observe that
\[
 |\xi|^m = |\xi|^{m-1}(\xi_1^2+ \cdots +
\xi_n^2)/|\xi|=|\xi|^{m-1}(\xi_1\cdot\frac{\xi_1}{|\xi|} + \cdots + \xi_n
\cdot \frac{\xi_n}{|\xi|})
\]
and that $\, \xi_j/|\xi|$ is the symbol of the $\, j$-th Riesz transform
being a bounded operator on $\, L^p,\; 1<p<\infty,$ and hence also on the
Lorentz spaces under consideration. Therefore,
when $\, \sigma=k \in {\mathbb N},$
 \[
\omega_k (D^m_R f,t)_{L_{p,r}} \lesssim \sup _{|\gamma|=m} \omega_k \left(
\frac{\partial^\gamma f}{\partial x^\gamma},t \right) _{L_{p,r}}
\]
and hence the assertion follows along the lines of the paper \cite{oscar}.
\hfill $\Box$

\smallskip
For the proof of
Proposition 1.3,
suppose that $\, \beta,\delta>0$ and that
$\, p,p^*$ and $\, \delta$ satisfy the assumptions. 
By Theorem 1.1 (iii) in \cite{see},
\begin{equation}\label{embBH}
B^{\,\delta}_{(p,r_0),q_1} \hookrightarrow L_{p^*,r_1} \,  \qquad \mbox{ if }
\quad 1 \le  q_1 \le r_1 \le \infty,\; \; 1\le r_0 \le \infty.
\end{equation}
Moreover, Theorem 1.6 (i, iii) in \cite{see} contains a version of the
Hardy-Littlewood-Sobolev theorem on fractional integration, which states that
\begin{equation}\label{embHB--}
H^{\beta +\delta}_{p,r_0} \hookrightarrow H^{\beta }_{p^*,r_1} \,
\qquad \mbox{ if }\quad 1\le r_0 \le r_1 \le \infty,\quad \beta\ge 0.
\end{equation}
The use of Holmstedt's formula completes the proof of (\ref{ulya}).

Concerning the proof of (\ref{kolya}), we need the embedding
\begin{equation}\label{embHB}
H^{\sigma +\delta}_{p,r_0} \hookrightarrow B^{\,\sigma}_{(p^*,r_1),q_0}\,
,\quad \mbox{ if }\, 1\le r_0 \le q_0 \le \infty,\; \; 1 \le r_1 \le \infty ,
\end{equation}
which holds by
\cite[Theorem 1.2 (iii)]{see}, and also embedding
(\ref{embBH}), which requires 
the additional restriction $\, q_1\le r_1\, .$ \hfill    $\Box$

\begin{remark}
Similarly to  Remark \ref{rem2.1}, we may derive that each of inequalities (\ref{deriv1})--(\ref{kolya}) implies the corresponding embedding.
For example, let (\ref{ulya}) be true. Since
$H^{\beta +\delta}_{p,r_0}=\{f\in L_{p,r_0}: \omega_{\beta +\delta}(f,u)_{L_{p,r_0}}\le Cu^{\beta +\delta} \}$,
inequality (\ref{ulya}) implies
$H^{\beta +\delta}_{p,r_0} \hookrightarrow H^{\beta }_{p^*,r_1}$, which is
(\ref{embHB--}). Likewise,  (\ref{kolya}) yields  (\ref{embHB}).
\end{remark}

\begin{remark}
 Let $n\in \mathbb{N},$ $1 < p < \infty,\; 0<\delta<n/p, \,1/p^*=1/p-\delta/n,$ and $\beta>0$.

(a) The combination of the Kolyada inequality
with the Marchaud inequality 
leads to a special case of the Ul'yanov
inequality. If  $\, 1\le r:=r_0=r_1=q_0=q_1\le \infty$ and $\, r \le \min \{ p^*,2\}$, then,
for all $ f\in L_{p,r}(\mathbb{R}^n)$,
\begin{equation}\label{koluly}
\omega_\beta (f,t)_{L_{p^*,r}} \lesssim \Big(
\int_{0}^{t}
[u^{-\delta}  \, \omega_{\beta +\delta}(f,u)_{L_{p,r}}]^{r}
\frac{du}{u}\Big)^{1/r}  \ \   \mbox{as } \ t\to 0+.
\end{equation}
This follows on  applying  to the left-hand side of (\ref{kolya})
 Marchaud inequality (\ref{march}), where we
replace $\, p$ by $\, p^*.$

(b) Similarly, if $\, 1\le r:= r_0=r_1=q_1\le \infty, $\, $r\ge
\max \{p^*,2\}$, and $\gamma >0,$ then the combination of Ul'yanov inequality (\ref{ulya})
and reverse Marchaud inequality (\ref{revmarch}) (where $p$ is replaced by $p^*$)
yields a special case of the Kolyada inequality, namely,
for all  $f \in L_{p,r}(\mathbb{R}^n),$
\begin{equation}\label{ulykol}
t^\beta
\Bigg (  \int_{t}^\infty \! \Big[ u^{-\beta}
\omega_{\beta +\gamma}(f,u)_{L_{p^*,r}}\Big]^{r}
\frac{du}{u}\Bigg)^{1/r} \!\!\!\! \lesssim
\!\Big( \int_{0}^{t}
[u^{-\delta}  \, \omega_{\beta +\delta}(f,u)_{L_{p,r}}]^{r}
\frac{du}{u}\Big)^{1/r}  \   \mbox{as } \ t\to 0+.
\end{equation}
Note that in the case $\, 0<\gamma <\delta$ the order of the modulus of
smoothness on the left-hand side is smaller than the one on the right-hand
side.

\end{remark}

\medskip
\subsection{Sharp Ul'yanov and Kolyada inequalities for
$p=1$
}
As it was mentioned above both (\ref{ulya}) and (\ref{kolya}) do not hold in general when $p=1$. However, under some additional conditions on parameters
 both results are still valid even in the Lorentz space setting.

\smallskip
{\sc Proposition 1.3\! $'$.} {\it
 Suppose  $n\in \N,$ $n\ge 2$, $1\le \delta <n, \,1/p^*=1-\delta/n,
\, 1 \le q_1\le r_1 \le \infty
$ and $\beta >0, \beta+\delta\in {\mathbb N}.$
}
\\[1mm]
{\rm (A)} 
{\it
Then, for all $\, t>0$ and for all $\, f \in
L_1({\mathbb R}^n),$
 }
\begin{equation}\label{ulya-p=1}
\omega_\beta(f,t)_{L_{p^*,r_1}} \lesssim \left( \int_{0}^{t}[u^{-\delta}
\omega_{\beta+\delta}(f,u)_{L_1} ]^{q_1}\frac{du}{u} \right)^{1/q_1}.
\end{equation}
\\[1mm]
{\rm (B)} 
 {\it If $1\le q_0
\le \infty$ then,
 for all $\, t>0$ and for all $\, f \in
L_1({\mathbb R}^n),$}
\[
t^\beta \left( \int_{t}^{\infty} [u^{-\beta} \omega_{\beta+\delta}(f,u)
_{L_{p^*,r_1}} ]^{q_0}\frac{du}{u}\right)^{1/q_0} \lesssim \left(
\int_{0}^{t}[u^{-\delta}\omega_{\beta+\delta} (f,u)_{L_1}]^{q_1}
\frac{du}{u} \right)^{1/q_1} .
\]
{
{\it Proof of  Proposition 1.3\! $'$(A)}}.
If $g\in H_{p^*,r_1}^\beta$, then in  light of  (\ref{wilmes}), for all $f\in {L_{p^*,r_1}}$ and all positive $t$, 
\begin{equation}\label{wil}
\omega_\beta(f,t)_{L_{p^*,r_1}} \approx K_0(f,t^\beta;
L_{p^*,r_1},H_{p^*,r_1}^\beta) \lesssim \|f-g\|_{p^*,r_1}+t^\beta
|g|_{H_{p^*,r_1}^\beta} .
\end{equation}
Now we  take into account  the following result by Alvino \cite{alv} (appeared in 1977, rediscovered  by
Poornima   \cite{poo} in 1983 and  by
Tartar \cite{tartar} in 1998)
\begin{equation}\label{alv}
\| h\|_{n/(n-1),1} \lesssim \sum_{j=1}^{n} \left\| \frac{\partial
h}{\partial x_j} \right\|_1 \, , \qquad n \ge 2.
\end{equation}
Together with  H\"ormander's multiplier criterion and  \cite[Theorem 1.6
(iii)]{see}, this yields
\begin{equation}\label{potemb}
W_1^{\beta+\delta} \hookrightarrow W_{n/(n-1),1}^{\beta+\delta-1}
= H_{n/(n-1),1}^{\beta+\delta-1} \hookrightarrow H^\beta_{p^*,1}
\hookrightarrow H^\beta_{p^*,r_1}\qquad\mbox{if} \quad r_1 \ge 1
\end{equation}
and for the corresponding seminorms we have, for all $g\in {W^{\beta+\delta}_1}$,
\begin{equation}\label{potembsn}
|g|_{H^\beta_{p^*,r_1}} \lesssim |g|_{H^\beta_{p^*,1}} \lesssim
|g|_{W^{\beta+\delta}_1} \, , \qquad
0<\frac{1}{p^*}=1-\frac{\delta}{n}\, , \quad r_1 \ge 1.
 \end{equation}
Note that using  Alvino's result, we  necessarily assume $\, \delta \ge 1$.

\medskip
By  \cite[Theorem 1.1 (iii)]{see}, the first embedding below is valid, the second one
is elementary and, therefore,
\begin{equation}\label{besemb}
B^\delta_{(1,1),q_1} \hookrightarrow L_{p^*,q_1} \hookrightarrow L_{p^*,r_1}\,
, \qquad \|f\|_{p^*,r_1} \lesssim |f|_{B^\delta_{(1,1),q_1}}\,
\end{equation}
for all $f\in {B^\delta_{(1,1),q_1}}$ if $1 \le
q_1 \le r_1 \le \infty \, .$

Applying  estimates (\ref{potembsn}), (\ref{besemb}), and
(\ref{wil}), we arrive at
\[
\omega_\beta(f,t)_{L_{p^*,r_1}} \lesssim |f-g|_{B^\delta_{(1,1),q_1}} + t^\beta
|g|_{W_1^{\beta +\delta}}
\]
for all $\, g \in W_1^{\beta +\delta}.$
Together with Holmstedt's formula, this  yields 
\begin{eqnarray*}
\omega_\beta(f,t)_{L_{p^*,r_1}} & \lesssim & K_0(f, t^\beta;
B^\delta_{(1,1),q_1}, W_1^{\beta +\delta})\\
 & \approx & \left( \int_{0}^{t^{\beta+\delta}}
[u^{-\delta/(\beta+\delta)} K_0(f,u;L_1, W_1^{\beta
+\delta})]^{q_1}\frac{du}{u} \right)^{1/q_1}\\
& \approx & \left( \int_{0}^{t} [u^{-\delta}
\omega_{\beta+\delta}(f,u)_{L_1}]^{q_1} \frac{du}{u} \right)^{1/q_1},
\end{eqnarray*}
where the condition
$\, \beta+\delta \in {\mathbb N}$ allows us to identify the resulting $\,
K_0$-functional with the classical modulus of smoothness in $\, L_1.$
$\hfill \Box$

{
{\it Proof of  Proposition 1.3\! $'$(B)}}.
Following the proof of  (\ref{kolya}), we need analogues of (\ref{embBH}) and (\ref{embHB}) for $ p=1$. In fact, in this case
(\ref{embBH}) holds  whenever $\, 1 \le q_1 \le r_1 \le \infty$ (see (\ref{besemb})).
Concerning (\ref{embHB}), we modify it by repeating the argument in (\ref{potemb}) to get
\[
W_1^{\beta+\delta} \hookrightarrow W_{n/(n-1),1}^{\beta+\delta-1}
= H_{n/(n-1),1}^{\beta+\delta-1}.
\]
Hence, applying  (\ref{embHB}) upon $\, H_{n/(n-1),1}^{\beta+\delta-1}$, under our assumptions, we arrive at
\[
W_1^{\beta+\delta} \hookrightarrow B^\beta_{(p^*,r_1),q_0}\, , \quad
\frac{1}{p^*}=1-\frac{\delta}{n}>0 , \; \delta\ge 1,\; \beta>0,\;
\delta + \beta \in {\mathbb N},\; 1\le q_0 \le \infty.
\]
By the Holmstedt formula,
\begin{eqnarray*}
I_{p^*} & := & t^{\beta/(\beta +\delta)} \left( \int_{t}^{\infty}
[u^{-\beta/(\beta + \delta)} K_0(f,u;L_{p^*,r_1},
H^{\beta+\delta}_{p^*,r_1})]^{q_0}\frac{du}{u} \right)^{1/q_0}\\
& \approx & K_0(f,t^{\beta/(\beta+\delta)}; L_{p^*,r_1},(L_{p^*,r_1},
H^{\beta+\delta}_{p^*,r_1})_{\beta/(\beta +\delta),\, q_0}\, ) \\
& \lesssim & \|f-g\| _{p^*,r_1}+ t^{\beta/(\beta+\delta)}
|g|_{B_{(p^*,r_1) ,\, q_0}^{\beta}}\\
& \lesssim & |f-g|_{B^\delta_{(1,1),q_1}} + t^{\beta/(\beta+\delta)}
  |g|_{W^{\beta +\delta}_1} \, ,\qquad \quad 1 \le q_1 \le r_1 \le \infty \, .
\end{eqnarray*}
Since this estimate holds for all $\, g \in W_1^{\beta +\delta}$, we have
\begin{eqnarray*}
I_{p^*}  & \lesssim  & {K}_0(f,t^{\beta/(\beta+\delta)};
(L_1,W_1^{\beta+\delta})_{\delta/(\beta+\delta),\, q_1} ,  W_1^{\beta
+\delta})\\
& \approx & \left( \int_{0}^{t} [u^{-\delta/(\beta+\delta)}
K_0(f,u;L_1,W_1^{\beta+\delta})] ^{q_1}\frac{du}{u  } \right)^{1/q_1}.
\end{eqnarray*}
Now simple substitutions, the characterizations of the $\,
K_0$-functionals via moduli of smoothness of integer order  give the assertion.
\hfill $\Box$

\begin{remark}
 Proposition 1.3 $'$  contains the corresponding results for Lebesgue spaces (for part (A), take $p^*=q_1= r_1$ and see
\cite{kolo, kolo1, tre4},
for part (B), take $p^*=q_1= r_1$, $q_0=1$ and see
\cite{kolo, kolo1, kolyada, tre4}).
We also note that even though (\ref{kol12}) does not hold in general for $p=1$ and $p^*<\infty$, it is still valid for $p=1$ and $p^*=\infty$ (\cite[Corollary~8.3]{kolo}), i.e., there holds
$$
\omega_k(f,t)_{L_{\infty}}
\lesssim \int_{0}^{t}u^{-n}\omega_{k+n}(f,u)_{L_{1}}
\frac{du}{u},
\qquad 
 k\in\mathbb{N}.$$
\end{remark}


\section{Notation and preliminaries }\label{section2}


Throughout the paper, we write ${\mathcal A}\lesssim 
{\mathcal B}$ (or
${\mathcal A}\gtrsim 
{\mathcal B}$) if $ {\mathcal A}\leq c\, {\mathcal
B}$ (or $c\,{\mathcal A}\geq {\mathcal B}$) for some positive constant $c$, 
 which depends only on nonessential  variables involved in the expressions
 ${\mathcal A}$ and ${\mathcal B}$, and
 ${\mathcal A}\approx {\mathcal B}$ if ${\mathcal A}\lesssim {\mathcal B}$ and ${\mathcal A}\gtrsim{\mathcal B}.$


In the whole paper 
the symbol $(\mathfrak{R},\mu)$ denotes a totally $\sigma$-finite measurable space with a non-atomic measure $\mu$,
and $\M(\mathfrak{R},\mu)$ is the set of all extended complex-valued
$\mu$-measurable functions on $\mathfrak{R}$.
By $\M^+(\mathfrak{R},\mu)$
we mean the family of
all non-negative functions from $\M(\mathfrak{R},\mu)$.
  When
$\mathfrak{R}$ is an interval $(a,b)\subseteq\mathbb{R}$ and $\mu$ is the Lebesgue measure on $(a,b)$, we denote
these sets by $\mathcal{M}(a,b)$ and
$\mathcal{M}^{+}(a,b),$ respectively.
Moreover,
by $\M^+(a,b;\downarrow)$ (and $\M^+(a,b;\uparrow)$)
we mean the subset of $\mathcal{M}^{+}(a,b)$ consisting of
all non-increasing (non-decreasing) functions on  $(a,b)$.
We denote by $\lambda_n$ the $n$-dimensional Lebesgue measure on $\mathbb{R}^n$.

 For two normed spaces
$\, X$ and $\, Y,$ we will use the notation $\, Y \hookrightarrow X$ if
$\, Y \subset X$ and $\, \| f\|_X \lesssim \| f\|_Y$ for all $\, f \in
Y.$

A normed linear space $X$ of functions from $\M(\mathfrak{R},\mu)$,  equipped with the
norm $\|\cdot\|_{X}$,  
is said to be a~{\it Banach function space} 
if the following four axioms hold:
\begin{itemize}


\item[(1)]\qquad $0\le g \le f$ $\mu$-a.e.\  implies
$\|g\|_{X} \le \|f\|_{X}$;
\item[(2)]\qquad $0\le f_n
\nearrow f$ $\mu$-a.e.\ implies $\|f_n\|_{X} \nearrow
\|f\|_{X}$;
\item[(3)]\qquad $\|\chi_E\|_{X}<\infty$ for every $E\subset\mathfrak{R}$ of finite measure;\footnote
{\ \ The symbol $\chi_E$ stands for the characteristic function of the set $E$.}

\item[(4)]\qquad if $\mu(E)<\infty,$
then there is a constant $C_E$  such that \newline
$ \qquad  \int_{
E}
|f(x)|\,d\mu(x) \le C_E \|f\|_{X} $
  for every $f\in X$.
  \end{itemize}
  Given
  a~{ Banach function space} $X$, which satisfies
 \begin{itemize}
\item[(5)]\qquad $\|f\|_{X} = \|g\|_{X}$ whenever $f\sp* = g\sp *$,\footnote{\ \  Recall that $f^*$ and $g^*$ denote
the non-increasing rearrangements of functions $f$ and $g$.} 
\end{itemize}
we obtain a {\it rearrangement-invariant Banach function space}
(shortly \textit{r.i.~space}).
Note that, by \cite[Chapter~2, Theorem 6.6]{besh} and \cite[Chapter~2, Theorem 2.7]{besh},
$L_1\cap L_\infty
\hookrightarrow X \hookrightarrow L_1+ L_\infty$ for any r.i. space $X$.

Given a  Banach function space $X$ on $(\mathfrak{R},\mu)$, the set
$$
X'=\left\{f\in\M(\mathfrak{R},\mu):\,\int_{\mathfrak{R}}|f(x)g(x)|\,d\mu<\infty\
\textup{for every } g\in X\right\},
$$
equipped with the norm
$$
\|f\|_{X'}=\sup_{\|g\|_{X}\leq1} \int_{\mathfrak{R}}|f(x)g(x)|\,d\mu,
$$
is called the \textit{associate space} of $X$. It turns out
that $X'$ is again a  Banach function space and that $X''=X$.
Furthermore, the \textit{H\"older inequality}
$$
\int_{\mathfrak{R}}|f(x)g(x)|\,d\mu\leq\|f\|_{X}\|g\|_{X'}
$$
holds for every $f$ and $g$ in $\M(\mathfrak{R},\mu)$. It will be useful to note that
\begin{equation}\label{E:duall}
\|f\|_{X}=\sup_{\|g\|_{X'}\leq1} \int_{\mathfrak{R}}|f(x)g(x)|\,d\mu.
\end{equation}

For every r.i.~space $X$ on $(\mathfrak{R},\mu)$, there exists 
an r.i.~space $\overline X$ over $((0,\infty), dt)$
 such that
$$
\|f\|_{X}=\|f\sp*\|_{\overline X}\quad \mbox{ for every } \ f\in X
$$
(cf.~\cite[Chapter~2, Theorem~4.10]{besh}). This space, equipped with the norm
$$
\|f\|_{\overline X}=\sup_{\|g\|_{X'}\leq1}
\int_0\sp{\infty}f\sp*(t)g\sp*(t)\,dt,
$$
is called the \textit{representation space} of $X$. 



A Banach space $F$ of real valued measurable functions defined on the measurable space  $(\mathfrak{R},\mu)$  is called \textit{ a Banach function lattice}  if its norm has the following property:
\[ |f(x)|\le |g(x)| \quad  \text{$\mu$-a.e.}, \quad   g\in F \ \ \ \Rightarrow \ \ \ f\in F \quad \text{and}\quad  \|f\|_F\le \|g\|_F.
\]
In this paper we will consider a Banach lattice  $F$  over a measurable space  $((0,\infty),dt/t)$, satisfying the condition
\begin{eqnarray} \label{cond5}
\Phi(1)<\infty,
\end{eqnarray}
where
$\Phi(x):=\|\min(x,\cdot)\|_{F}$ for all $x\in (0,\infty)$. (The function $\Phi$  is sometimes called the fundamental function of the lattice $F$.)
Note that 
 $\Phi$ is a quasiconcave function on $(0,\infty)$,
which means that $\Phi\in\mathcal{M}^{+}((0,\infty);\uparrow)$ and $\frac{\Phi}{Id}\in \mathcal{M}^{+}((0,\infty);\downarrow)$  
 (here $Id$ stands for the identity map on $(0,\infty)$). Condition (\ref{cond5}) implies that
$\Phi(x)<\infty$ for any $x\in (0,\infty)$, moreover, $\Phi\in C(0,\infty)$ (cf.  \cite [Remark~2.1.2]{evgo}).

Let $\, (X, Y)$ be a compatible couple of Banach spaces (cf., \cite[p. 310]{besh}).
 The $K$-functional is defined for each $f\in X+Y$ and $t>0$ by
\begin{eqnarray} \label{k-f0--}
K(f,t;X, Y):= \inf _{f=f_1+f_2} \Big(\| f_1\|_{X} + t
\|f_2\|_{Y} \Big),
\end{eqnarray}
where the infimum extends over all representation $f=f_1+f_2$ 
  with
$f_1 \in X$ and $f_2 \in Y$. As a~function of $t$, $K(f,t;X,Y)$
 is \textit{quasiconcave} on $(0,\infty)$. 

Similarly, we define, for each $f\in X+ Y$ and $t>0,$
\begin{eqnarray} \label{k-f0}
K_0(f,t;X, Y):= \inf _{f=f_1+f_2} \Big(\| f_1\|_{X} + t
|f_2|_{Y} \Big)
\end{eqnarray}
and
\begin{eqnarray} \label{k-f1}
K_1(f,t;X, Y):= \inf _{f=f_1+f_2} \Big( | f_1|_{X} + t
|f_2|_{Y} \Big),
\end{eqnarray}
where $| \cdot |_{X}$ and $|\cdot |_{Y}$ are seminorms on $X$ and $Y$.

If $(X,Y)$ is a compatible couple of Banach spaces and
$F$ is a Banach lattice, then we
define the space $(X,Y)_F$ to be the set   of all $f\in X+Y$ for which the norm
\[ \|f\|_{(X,Y)_F}=  \|K(f, \cdot; X,Y)\|_{F}
\]
is finite. Note that if $1\le r < \infty$, $\theta \in (0,1)$ and the Banach lattice $F$ is the set of
all functions $h \in\mathcal{M}(0,\infty)$ such that
$$
\|h\|_F:=\Big(\int_0^\infty\Big(s^{-\theta} |h(s)|\Big)^r \frac{ds}{s}\Big)^{1/r} <\infty,
$$
then the space $(X,Y)_F$ coincides with the classical space $(X,Y)_{\theta, r}$ defined, e.g.,
in \cite [p.~299] {besh}.

We will also work with more general classes of functions, which are not linear.
Let $\rho$ be a functional on  $\M^+({\R^n},\lambda_n)$ satisfying

(N1) $\rho(f)\ge 0$ for any $f\in \M^+({\R^n},\lambda_n)$ and  $\rho(f)=0 \Leftrightarrow f=0$, $\lambda_n$-a.e.,

(N2) $\rho(\alpha f)=\alpha\rho(f)$ for any $f\in \M^+({\R^n},\lambda_n)$ and $\alpha\ge 0$.
\\
Such a functional is called a gage and
 the collection
$$
X=X(\R^n)=X({\R^n},\lambda_n)=\{f\in \M^+({\R^n},\lambda_n): \rho(f)<\infty\}
$$ is said  {\it a $($function$)$ gaged cone} (cf. \cite{cw-peetre}). Moreover, we put
$$\|f\|_X:=\rho(f),\quad f\in{X}.$$
An {associate space} of a gaged cone $X$ is defined in the same way as for Banach function spaces.

If $X$ is a gaged cone, then the functional $|\cdot|_X:X\to\R$
is called a  semi-gage on $X$ provided that
the
functional $|\cdot|_X$ is non-negative and positively homogeneous  on $X$.

Given two function gaged cones $X$ and $Y$, the embedding $Y\hookrightarrow X$ means that $Y\subset X$ and $\|f\|_X\lesssim \|f\|_Y$ for all $f\in Y.$

A pair of function gaged cones $(X,Y)$ is said a compatible couple of function  gaged cones if there is some Hausdorff topological vector space, say $Z$,
in which each of $X$ and $Y$ is continuously embedded. Given  a compatible couple  $(X,Y)$ of function  gaged cones, the $K$-functionals
$K(f,t;X, Y), K_0(f,t;X, Y),$
and
$K_1(f,t;X, Y)$
are defined
analogously to (\ref{k-f0--})--(\ref{k-f1}). Moreover, if $F$ is a Banach lattice over a measure space $((0,\infty),dt/t)$
satisfying  (\ref{cond5}), then the space $(X,Y)_F$ is defined analogously to the case when $(X,Y)$ is a compatible couple of Banach spaces.

In this paper we work with  function gaged cones being the subsets of
 $
L_1(\R^n)+L_\infty(\R^n)$.





Given $k\in\mathbb N$ and a Banach function space $X=X(\R^n)$, we denote by $W^kX$ the corresponding Sobolev space,
 that is,
the space of all functions on $\R^n$ whose distributional derivatives $D^\alpha f$, $|\alpha|\le k$, belong to $X$.
This space is equipped with the norm $$
\|f\|_{W^kX}
:=
\|f\|_{X}+|f|_{W^kX}:=
\|f\|_{X}+\sum\limits_{k=|\alpha|}\|D^{\,\alpha} f
\|_X.$$
Note that $W^kX=A^kX$, where $A$ is the Sobolev integral operator; see, for example, the representation theorem in \cite[Section~3.4]{bur}.
If $X$ is a function  gaged cone, then the Sobolev class  $W^kX$ is defined similarly.

We are going to use the classical equivalence between the $K$-functional $K_0$ and modulus of smoothness: for any $k\in\mathbb N$ and
an r.i. Banach function  space $X$, one has
\begin{equation}\label{mod-equi}
\omega_k(f,t)_X\approx K_0(f,t^k; X,W^{k}X)\qquad \mbox{for all}\ \  t> 0\ \mbox{and} \ \  f\in X
\end{equation}
provided that in the space $W^{k}X$ we choose the seminorm 
$
|f|_{W^{k}X}:=\sum\limits_{k=|\alpha|}\|D^{\,\alpha} f
\|_X$.
 The proof follows the same reasoning as the one given for $X=L_p$ in \cite[pp. 339--341]{besh}.

Let $-\infty \le a<b \le +\infty$ 
and let $\xi:(a,b)\to \mathbb R$ be  a non-decreasing function on $(a,b).$
Put $
\xi(a)=\lim_{t\to a+}\xi(t)$ and $
\xi(b)=\lim_{t\to b-}\xi(t)$.
{\it The generalized reverse function} $\mathrm{R}\xi$ of $\xi$ is defined by
\begin{eqnarray*} 
(\mathrm{R}\xi)(t):= \inf \Big\{
\tau\in (a,b): \,
\xi (\tau)>t
\Big\}\ \  \mbox{ for all } \ t\in (\xi(a),\xi(b)).
\end{eqnarray*}
The following properties of the generalized reverse function can be easily verified.
\begin{lemma}\label{lemma-reverse}
If the function $\xi$ given above is left continuous on $(a,b)$, then
$$\xi\big((\mathrm{R}\xi)(t)\big)\le t \ \ \mbox{ for any } \ t\in (\xi(a),\xi(b))$$
and
$$t\le(\mathrm{R}\xi)(\xi(t)) \ \ \mbox{ for any } \ t\in (a,b).$$
 Moreover, if $\xi\in C((a,b)),$  then
$$\xi\big((\mathrm{R}\xi)(t)\big)= t  \ \ \mbox{ for any } \ t\in (\xi(a),\xi(b)).$$
\end{lemma}
We note that 
Lemma \ref{lemma-reverse} 
does not hold without the assumption that $\xi$ is left continuous.
Moreover,   an analogue of $\xi\big((\mathrm{R}\xi)(t)\big)= t$, namely
$(\mathrm{R}\xi)(\xi(t))=t$ for any $t\in (a,b)$, need not hold even if $\xi\in C((a,b)).$

If $(a,b)\subset \mathbb R$ and $p \in (0,\infty]$, then the symbol $\|\cdot\|_{p,(a,b)}$ stands for the quasinorm
in the Lebesgue space $L_p((a,b)).$

As usual, for $ p\in [1,\infty]$, we define $p'$ by $1/p+1/p'=1$.
Throughout the paper we use the abbreviations $\text{LHS}(*)$
($\text{RHS}(*)$) 
for the
left- (right-) hand side of the relation
$(*)$.

\section {General
inequalities for  K-functionals }\label{holm-section}

\subsection{Holmstedt-type  formulas}

The next theorem is a folklore  in some way and it can be considered  as an abstract form of the limiting cases of the Holmstedt-type formulas
 (see, e.g., \cite[Corollary~2.3, p. 310 and p. 430]{besh} and \cite[p. 466]{brud}).
  Since  we have not been able to find an explicit reference of the needed general form (cf. \cite{karadzhov, mas}), we prove it below.
 The importance of this result can be seen in, e.g., \cite{ran}.

\begin{theorem} \label{holmsted-f}Let $(X_0,X_1)$ be a compatible couple of Banach function spaces.

{\textnormal{(A)}}
Let $F_0$ be a  Banach lattice over $((0,\infty),dt/t)$.
Assume  that the function
$\Xi(t):=\|\min(\cdot,t)\|_{F_0}$, $t \in (0,\infty)$,
satisfies $\Xi(1)<\infty.$
If  $\phi$ is the generalized reverse function of
$
\Xi,
$
then
\begin{align} \label{A1}
K(f,t; (X_0,X_1)_{F_0},X_1)&\approx
\|
K(f,s; X_0,X_1) \chi_{(0,\phi(t))}(s)
\|_{F_0} \\
&\qquad+K(f,\phi(t); X_0,X_1) \|
\chi_{(\phi(t),\infty)}(s)\nonumber
\|_{F_0}
\end{align}
for all $t \in (\Xi(0), \Xi(\infty))$ and $f \in (X_0,X_1)_{F_0}+X_1.$

{\textnormal{(B)}}
Let $F_1$ be a  Banach lattice over $((0,\infty),dt/t)$.
Assume  that the function
$\Theta(t):=
t/
 \|\min(\cdot,t)\|_{F_1}$, $t \in (0,\infty)$,
satisfies $\Theta(1)<\infty.$
If  $\psi$ is the generalized reverse function of
$
\Theta,
$
then
\begin{align} \label{A2}
K(f,t; X_0, (X_0,X_1)_{F_1})&\approx
t\frac{K(f,\psi(t); X_0,X_1)}{\psi(t)}\|
 s\,\chi_{(0,\psi(t))}(s)
\|_{F_1}\\&\qquad+t
\|
K(f,s; X_0,X_1) \chi_{(\psi(t),\infty)}(s)\nonumber
\|_{F_1}
\end{align}
for all $t \in (\Theta(0), \Theta(\infty))$ and 
$f \in X_0 + (X_0,X_1)_{F_1}.$
\end{theorem}


\begin{remark}\label{remark-} (i) Formulas (\ref{A1}) and (\ref{A2}) remain valid for $K$-functionals given by (\ref{k-f0}) and (\ref{k-f1}).

(ii) By Theorem \ref{holmsted-f} estimate \eqref{A1} holds for all $t \in (\Xi(0), \Xi(\infty))$ and $f \in (X_0,X_1)_{F_0}+X_1$,
or equivalently, for all $t \in (\Xi(0), \Xi(\infty))$ and $f \in X$ for which $\text{RHS}\eqref{A1}$ is finite. Similar remark can be made about equivalence  \eqref{A2}.
\end{remark}

\begin{proof}[Proof of Theorem \ref{holmsted-f}]
We start with (A).  As the function
$\Xi$
 is quasiconcave,  it is continuous and hence
$\Xi(\phi(t))=\|\min(s,\phi(t))\|_{F_0}=t$ for any $t \in (\Xi(0), \Xi(\infty))$ by Lemma~\ref{lemma-reverse}.
If $f=f_0+f_1$, where $f_0\in (X_0,X_1)_{F_0}$ and $f_1\in X_1$,
then, for all $t \in (\Xi(0), \Xi(\infty))$,
\begin{align*}
\|
K(f,s; X_0,X_1) \chi_{(0,\phi(t))}(s)
\|_{F_0}
&\le
\|
K(f_0,s; X_0,X_1) \chi_{(0,\phi(t))}(s)
\|_{F_0}
\\
& \qquad+
\|
K(f_1,s; X_0,X_1) \chi_{(0,\phi(t))}(s)
\|_{F_0}
\\
&\le
\|
K(f_0,s; X_0,X_1)
\|_{F_0}
+
\|
s \chi_{(0,\phi(t))}(s)
\|_{F_0}
\|f_1\|_{X_1}
\\
&\le
\|f_0\|_{(X_0,X_1)_{F_0}}
+
\|\min(s,\phi(t))\|_{F_0}\|f_1\|_{X_1}\\
&=
\|f_0\|_{(X_0,X_1)_{F_0}}
+
t\|f_1\|_{X_1}
\end{align*}
and
\begin{align*}
K(f,\phi(t); X_0,X_1)\| \chi_{(\phi(t),\infty)}(s)
\|_{F_0}
&\le
K(f_0,\phi(t); X_0,X_1) \|\chi_{(\phi(t),\infty)}(s)
\|_{F_0}\\
&\qquad+
K(f_1,\phi(t); X_0,X_1)\| \chi_{(\phi(t),\infty)}(s)
\|_{F_0}
\\
&\le
\|
K(f_0,s; X_0,X_1)
\|_{F_0}
+ \phi(t) \|f_1\|_{X_1}\| \chi_{(\phi(t),\infty)}(s)
\|_{F_0}
\\
&\le
\|f_0\|_{(X_0,X_1)_{F_0}}
+
\|\min(s,\phi(t))\|_{F_0}\|f_1\|_{X_1}
\\
&=
\|f_0\|_{(X_0,X_1)_{F_0}}
+
t\|f_1\|_{X_1}.
\end{align*}
Thus, taking the infimum over all decompositions  $f=f_0+f_1$ of the function $f$, with $f_0 \in (X_0,X_1)_{F_0}$ and
$f_1 \in X_1,$
we arrive at the estimate  LHS(\ref{A1})$\gtrsim$RHS(\ref{A1}).  

To prove the opposite estimate,
take $t \in (\Xi(0), \Xi(\infty))$ and
suppose that
 $f=f_0+f_1$, with $f_0\in X_0$, $f_1\in X_1$,
 be such a representation that 
$$\|f_0\|_{X_0}+\phi(t)\|f_1\|_{X_1}\le 2K(f,\phi(t);X_0,X_1).
$$
Since, for all $s>0$,
$$K(f_0,s;X_0,X_1)\le \|f_0\|_{X_0}\le 2K(f,\phi(t);X_0,X_1)
$$
and
$$\frac{K(f_1,s;X_0,X_1)}{s}\le \|f_1\|_{X_1}\le   \frac{2 }{\phi(t)}K(f,\phi(t);X_0,X_1),
$$
 we get, for all $f \in (X_0,X_1)_{F_0}+X_1,$

\begin{align*}
K(f,t; (X_0,X_1)_{F_0},X_1)&\le\|f_0\|_{(X_0,X_1)_{F_0}} + t \|f_1\|_{X_1}
\\
&\lesssim
\| K(f_0,s; X_0,X_1)\|_{F_0} +t  \frac{K(f,\phi(t); X_0,X_1) }{\phi(t)}\\
&\lesssim \| K(f_0,s; X_0,X_1)\chi_{(0,\phi(t))}(s)\|_{F_0} +\| K(f_0,s; X_0,X_1)\chi_{(\phi(t),\infty)}(s)\|_{F_0} \\
&\qquad+t  \frac{K(f,\phi(t); X_0,X_1) }{\phi(t)}=:J_1+J_2+J_3.
\end{align*}

As  $f_0=f-f_1$, we obtain
\begin{align*}
J_1
&\le
\|K(f,s; X_0,X_1)\chi_{(0,\phi(t))}(s)\|_{F_0} 
+\|K(f_1,s; X_0,X_1)\chi_{(0,\phi(t))}(s)\|_{F_0}\\
&\le  \|K(f,s; X_0,X_1)\chi_{(0,\phi(t))}(s)\|_{F_0} + \|f_1\|_{X_1}\|s\chi_{(0,\phi(t))}(s)\|_{F_0}\\
&\lesssim\|K(f,s; X_0,X_1)\chi_{(0,\phi(t))}(s)\|_{F_0} +\frac{K(f,\phi(t); X_0,X_1) }{\phi(t)}\|s\chi_{(0,\phi(t))}(s)\|_{F_0}\\
&\lesssim \|K(f,s; X_0,X_1)\chi_{(0,\phi(t))}(s)\|_{F_0}
\end{align*}
and
\begin{align*}
J_2&\lesssim \|f_0\|_{X_0}\| \chi_{(\phi(t),\infty)}(s)\|_{F_0}\lesssim
K(f,\phi(t); X_0,X_1)\|\chi_{(\phi(t),\infty)}(s)\|_{F_0}.
\end{align*}
Since $t=\Xi(\phi(t))\le \|
s\,\chi_{(0,\phi(t))}(s)
\|_{F_0}+\phi(t)\|
\chi_{(\phi(t),\infty)}(s)
\|_{F_0}$, we get
\begin{align*}
J_3
&\le
\Big(\|
s\,\chi_{(0,\phi(t))}(s)
\|_{F_0}+\phi(t)\|
\chi_{(\phi(t),\infty)}(s)
\|_{F_0}\Big)\frac{K(f,\phi(t); X_0,X_1) }{\phi(t)}
\\
&\le
\|
K(f,s; X_0,X_1) \chi_{(0,\phi(t))}(s)
\|_{F_0}
+
 K(f,\phi(t); X_0,X_1) \|
\chi_{(\phi(t),\infty)}(s)
\|_{F_0}.
\end{align*}
Consequently, for all $t \in (\Xi(0), \Xi(\infty))$ and $f \in (X_0,X_1)_{F_0}+X_1,$
$$
 K(f,t; (X_0,X_1)_{F_0},X_1)\lesssim
\|
K(f,s; X_0,X_1) \chi_{(0,\phi(t))}(s)
\|_{F_0} +K(f,\phi(t); X_0,X_1) \|
\chi_{(\phi(t),\infty)}(s)
\|_{F_0}.
$$

To prove part (B), we notice  that
(cf. \cite[Chapter~V, Prop. 1.2]{besh})
\begin{align*}
K(f,t; X_0, (X_0,X_1)_{F_1})&=
t K(f,1/t; (X_0,X_1)_{F_1},X_0)
=
t K(f,1/t; (X_1,X_0)_{\widetilde{{F_1}}},X_0),
\end{align*}
 where
 $\widetilde{{F_1}}=\{f: tf(1/t)\in {F_1} \}$
and
 $\|f\|_{\widetilde{{F_1}}}=\|tf(1/t)\|_{{F_1}}.$
Now we apply part (A) and the reverse preceding substitutions to arrive at the statement.
 \end{proof}

\subsection{Inequalities for $K$-functionals
involving the potential-type operators}

Let $\{A^\tau\}_{\tau\in \mathfrak{M}}$, where
  $ \mathfrak{M}=\{\tau: 0\le \tau<\tau_0\}$ or $\mathfrak{M}=\{k\in\mathbb{N}_0: k<\tau_0\}$ for some $\tau_0\in (0,\infty)$,
 be a family of linear operators defined on
 $L_1(\R^n)+L_\infty(\R^n)$ satisfying
\begin{enumerate}
\item[(P1)] \;$A^\tau: X\to X$  \;for any \; $\tau\in \mathfrak{M}$ and for any function gaged cone  
$X\subset L_1(\R^n)+L_\infty(\R^n)$;
   \item[(P2)] \;$A^0 X=X$ for   any function gaged cone  
$X\subset L_1(\R^n)+L_\infty(\R^n)$; 
   \item[(P3)] \;$A^\tau(A^\sigma X)=A^\sigma(A^\tau X)=A^{\tau+\sigma} X$  \;
 for any function gaged cone  
$X\subset L_1(\R^n)+L_\infty(\R^n)$
and  for any $\sigma,\tau,\sigma+\tau\in \mathfrak{M}$.
 \end{enumerate}
 Here $ A^\tau X$,  $\tau\in \mathfrak{M},$ is the range of $ A^\tau$ equipped with the gage (or norm)
$$
\|f\|_{A^\tau X} =\|f\|_X+|f|_{A^\tau X},\quad\mbox{where}\quad
 |f|_{A^\tau X}=\inf\{\|g\|_X: f=A^\tau g\}.
$$

\begin{theorem} \label{emb-ul-holm} {\textnormal{(Ul'yanov-type inequalities)}}
Assume that $X$ is an r.i. Banach function space 
  and
$ Y,Z\subset L_1+L_\infty$ are function gaged cones.


Let
 \begin{equation}\label{emb1}
A^{\sigma+\tau}X\hookrightarrow Y,\quad \quad  A^{\sigma}X\hookrightarrow Z,\quad \mbox{for some $\tau>0$ and $\sigma\ge 0$.}
\end{equation}

{\textnormal{(A)}} Let $F_0$ 
 be a Banach lattice  over $((0,\infty),dt/t)$ satisfying
\begin{equation}\label{emb2}
(X,Y)_{F_0}\hookrightarrow  Z.
\end{equation}
Assume that the function
$\Xi_0(t):=\|\min(\cdot,t)\|_{F_0},$ 
$t \in (0,\infty)$,
is such that $\Xi_0(1)<\infty$. 
If  $\phi_0$  
is the generalized reverse function of $\Xi_0$, 
then
\begin{align} \label{emb3}
K(f,t; Z,A^\tau Z) &\lesssim
\|
K(f,s; X,A^{\sigma+\tau}X)
\chi_{(0,\phi_0(t))}(s)
\|_{F_0}
\\&\qquad +K(f,\phi_0(t); X,A^{\sigma+\tau}X)\| \chi_{(\phi_0(t),\infty)}(s)\|_{F_0}\nonumber
\end{align}
for all $t \in (\Xi_0(0), \Xi_0(\infty))$ and $f \in X$ {\rm (}for which ${\rm RHS}\eqref{emb3}$ is finite{\rm )}.

{\textnormal{(B)}} Let $F_1$
 be a Banach lattice  over $((0,\infty),dt/t)$ satisfying
\begin{equation}\label{emb2_1}
Z\hookrightarrow (X,Y)_{F_1}=:V.
\end{equation}
Assume that the function
 $\Xi_1(t):=\|\min(\cdot,t)\|_{F_1}$ ,
$t \in (0,\infty)$,
is such that  $\Xi_1(1)<\infty.$
If  $\phi_1$
is the generalized reverse function of  $\Xi_1$,
then
\begin{align} \label{emb3'}
K(f,t; V,A^\tau V) &\lesssim
\|
K(f,s; X,A^{\sigma+\tau}X)
\chi_{(0,\phi_1(t))}(s)
\|_{F_1}
\\&\qquad +K(f,\phi_1(t); X,A^{\sigma+\tau}X)\| \chi_{(\phi_1(t),\infty)}(s)\|_{F_1}\nonumber
\end{align}
for all $t \in (\Xi_1(0), \Xi_1(\infty))$ and $f \in X$ {\rm (}for which ${\rm RHS}\eqref{emb3'}$ is finite{\rm )}.
\end{theorem}


\begin{remark} (i) It is clear from the proof  that inequality  (\ref{emb3}) holds provided that
$$
A^{\sigma+\tau}X\hookrightarrow Y\quad \mbox{for some $\tau>0$ and $\sigma\ge 0$}
$$
and
$$
A^{\sigma}X\hookrightarrow (X,Y)_{F_0}=:Z.
$$

(ii)  A different, abstract approach to Ul'yanov inequalities,
based on semi-groups of linear equibounded operators in Banach spaces, is given in \cite{trwe2}.
\end{remark}

\begin{proof}[Proof of Theorem \ref{emb-ul-holm}]
The property (P3) of operators $A^\tau$
and the second embedding in (\ref{emb1})
 imply that
\begin{equation}\label{emb2'}
A^{\sigma+\tau}X\hookrightarrow A^\tau Z.
\end{equation}
Further,
using
 the first embedding in (\ref{emb1}) and (\ref{emb2}),
we get
$$(X, A^{\sigma+\tau}X)_{F_0} \hookrightarrow
(X, Y)_{F_0}\hookrightarrow Z.
$$
This, (\ref{emb2'}), and Theorem \ref{holmsted-f} (A) (see also Remark \ref{remark-} (ii)) yield, for all $t \in (\Xi_0(0), \Xi_0(\infty))$ and $f\in X$,
\begin{align*}
 \begin{split}
K(f,t; Z,A^\tau Z)
& \lesssim
K(f,t; (X,A^{\sigma+\tau}X)_{F_0},A^{\sigma+\tau}X)
\\
&\approx
\|
K(f,s; X,A^{\sigma+\tau}X) \chi_{(0,\phi_0(t))}(s)
\|_{F_0} \\
&\qquad+ K(f,\phi_0(t); X,A^{\sigma+\tau}X)\| \chi_{(\phi_0(t),\infty)}(s)\|_{F_0},
\end{split}
\end{align*}
and (\ref{emb3}) is proved.

To obtain  \eqref{emb3'}, using
(\ref{emb2'}) and  (\ref{emb2_1}),
we arrive at
$$
A^{\sigma+\tau}X\hookrightarrow A^\tau V.
$$
Moreover, applying the first embedding in
 (\ref{emb1}) and definition of  $V$,
we obtain
$$(X, A^{\sigma+\tau}X)_{F_1} \hookrightarrow
(X, Y)_{F_1} =V.
$$
Consequently, for all $t>0,$
$$
%
K(f,t; V,A^\tau V) \lesssim
K(f,t; (X, A^{\sigma+\tau}X)_{F_1},A^{\sigma+\tau}X),
$$
which, together with  Theorem \ref{holmsted-f} (A) (and Remark \ref{remark-} (ii)), yields  \eqref{emb3'}.
\end{proof}

Using part (B) of Theorem  \ref{holmsted-f}, one can prove the following results (Marchaud and reverse Marchaud-type 
 inequalities).

\begin{theorem} \label{emb-marchaud-holm}
Assume that $X$ is an r.i. Banach function space. 
 Let $F_1$ be a  Banach lattice over $((0,\infty),dt/t)$.
Assume  that the function
$\Theta(t):=
t/
 \|\min(\cdot,t)\|_{F_1}$, $t \in (0,\infty)$,
satisfies $\Theta(1)<\infty$
and that $\psi$ is the generalized reverse function of
$
\Theta.
$
 \vskip 0.1cm

{\textnormal{ (A)}}  {\rm (Marchaud-type inequality)}
If \begin{equation}\label{emb22}
(X,A^{\sigma+\tau}X)_{F_1} \hookrightarrow
A^{\tau}X,\quad\mbox{with some $\tau,\sigma>0$,}
\end{equation}
then
\begin{align}
\label{emb23--}
K(f,t; X,A^\tau X) &\lesssim
t\frac{K(f,\psi(t); X,A^{\sigma+\tau}X)}{\psi(t)}
\|
 s\,\chi_{(0,\psi(t))}(s)
\|_{F_1}
\\
&
\qquad+
t\|
K(f,s; X,A^{\sigma+\tau}X)
\chi_{(\psi(t),\infty)}(s)\|_{F_1}\nonumber
\end{align}
for all $t \in (\Theta(0), \Theta(\infty))$ and $f \in X$ {\rm (}for which ${\rm RHS}\eqref{emb23--}$ is finite{\rm )}.
 \vskip 0.1cm

{\textnormal{ (B)}} {\rm (Reverse Marchaud-type 
inequality)}
If
\begin{equation}\label{emb22>}
A^{\tau}X \hookrightarrow
(X,A^{\sigma+\tau}X)_{F_1},\quad \mbox{with some $\tau,\sigma>0$,
}
\end{equation}
then
\begin{eqnarray}
\label{emb23-->}
\quad t\frac{K(f,\psi(t); X,A^{\sigma+\tau}X)}{\psi(t)}
\|
 s\,\chi_{(0,\psi(t))}(s)
\|_{F_1}\!\!\!\!
&+
t\|
K(f,s; X,A^{\sigma+\tau}X)
\chi_{(\psi(t),\infty)}(s)
 \|_{F_1}
\\
&\qquad \qquad \qquad \qquad \lesssim
K(f,t; X,A^\tau X) \notag
\end{eqnarray}
for all $t \in (\Theta(0), \Theta(\infty))$ and $f \in X$ {\rm (}for which ${\rm RHS}\eqref{emb23-->}$ is finite{\rm )}.
\end{theorem}

\begin{proof}
To prove (A),  we obtain, by (\ref{emb22})
and Theorem  \ref{holmsted-f} (B) (see also Remark~\ref{remark-}~(ii)),
\begin{align*}
K(f,t; X,A^\tau X) &\lesssim
K(f,t; X,(X,A^{\sigma+\tau}X)_{F_1})
\\&\approx
t\frac{K(f,\psi(t); X,A^{\sigma+\tau}X)}{\psi(t)}
\|
 s\,\chi_{(0,\psi(t))}(s)
\|_{F_1}
\\
&\qquad+t\|
K(f,s; X,A^{\sigma+\tau}X)
\chi_{(\psi(t),\infty)}(s) \|_{F_1}
\end{align*}
for any $t \in (\Theta(0), \Theta(\infty))$ and $f \in X$.

In part (B)  embedding (\ref{emb22>}) is  reverse to (\ref{emb22}), therefore the above inequality sign is also reverse.
\end{proof}

Combining  parts (A) and (B) of Theorem \ref{holmsted-f}, we  obtain  the following result.

\begin{theorem} \label{emb-kol-holm} {\rm (Kolyada-type inequality)}
Assume that $X$ and $Z$, $Z\subset X$, are r.i. Banach function spaces.
Let $F_0, F_1$ be   Banach lattices over $((0,\infty),dt/t)$ satisfying, for some $\tau>0$ and $\sigma\ge 0$,
\begin{equation}\label{+emb1}
(X,A^{\tau+\sigma}X)_{F_0}\hookrightarrow Z 
\end{equation}
and
\begin{equation}\label{+emb2}
A^{\tau+\sigma}X\hookrightarrow (Z,A^{\tau}Z)_{F_1}.
\end{equation}
Assume  that the functions
\,$\Xi(t):=\|\min(\cdot,t)\|_{F_0}$ and
$\Theta(t):=
t/
 \|\min(\cdot,t)\|_{F_1}$, $t \in (0,\infty)$,
satisfy $\Xi(1)<\infty$ and  $\Theta(1)<\infty$.
If  $\phi$ and $\psi$ are the generalized reverse functions of $\Xi$ and
$
\Theta,
$ respectively,
then
\begin{align}
\label{+emb3}
t\frac{K(f,\psi(t); Z,A^{\tau}Z)}{\psi(t)}
\|
 s\,&\chi_{(0,\psi(t))}(s)
\|_{F_1}
+
t\|
K(f,s; Z,A^{\tau}Z)
\chi_{(\psi(t),\infty)}(s)
\|_{F_1}
\\ \nonumber
&\ \ \quad\lesssim
\|
K(f,s; X,A^{\tau+\sigma}X)
\chi_{(0,\phi(t))}(s)
\|_{F_0}\\ \nonumber
&\ \ \quad\qquad+
K(f,\phi(t); X,A^{\tau+\sigma}X)\| \chi_{(\phi(t),\infty)}(s)\|_{F_0}
\end{align}
for all  $t \in (\Xi(0), \Xi(\infty))\cap(\Theta(0), \Theta(\infty))$
 and  $f \in X$ {\rm (}for which ${\rm RHS}\eqref {+emb3}$ is finite{\rm )}.\end{theorem}
\begin{proof}
Taking into account
(\ref{+emb1}) and (\ref{+emb2}), we get
$$
K(f,t; Z,(Z,A^{\tau}Z)_{F_1})\lesssim
K(f,t; (X,A^{\tau+\sigma}X)_{F_0},A^{\tau+\sigma}X)\ \  \mbox{ for all }\  t>0.
$$
To complete the proof, note that,
by Theorem \ref{holmsted-f} (B),
$$K(f,t; Z,(Z,A^{\tau}Z)_{F_1})\approx \text{LHS}(\ref{+emb3}) \ \ \mbox{ for all }\   t \in (\Theta(0), \Theta(\infty))$$
and, by Theorem \ref{holmsted-f} (A),
$$K(f,t; (X,A^{\tau+\sigma}X)_{F_0},A^{\tau+\sigma}X)\approx \text{RHS}(\ref{+emb3}) \ \ \mbox{ for all }\  t \in (\Xi(0), \Xi(\infty)) .$$
\end{proof}
\begin{remark}\label{remark--}
(i) Note that Theorems \ref{emb-ul-holm}, \ref{emb-marchaud-holm}, and \ref{emb-kol-holm} remain true if the $K$-functional $K$
is replaced by the $K$-functional $K_0$ or by the $K$-functional $K_1$ given by \eqref{k-f0} or by \eqref{k-f1}.

(ii)
Theorems
\ref{holmsted-f},
\ref{emb-ul-holm}, \ref{emb-marchaud-holm}, and \ref{emb-kol-holm} are true if the Banach function spaces are replaced by function gaged  cones.

\end{remark}

To give a flavor of how to use Theorems \ref{emb-ul-holm}, \ref{emb-marchaud-holm}, and \ref{emb-kol-holm}, we present the following
examples on the classical Ul'yanov inequality (\ref{ulul}) and  sharp Ul'yanov inequality in the Lorentz setting, cf. Proposition 1.3.

\begin{example}
We obtain the following extension of the classical Ul'yanov inequality (\ref{ulul}):

{\it
If  $1 \le p<\infty$, $k, n\in {\mathbb N},$ $0<\delta< \min (k,  n/p)$,  and $1/p^*=1/p-\delta/n,$ then
  \begin{equation}\label{ulul---}
t^{k-\delta}\sup\limits_{t\le u<1}\frac{\omega_k(f,u)_{L_{p^*}}}
{u^{k-\delta}}\lesssim
\left( \int_{0}^t \Big[ u^{-\delta}
\omega_{k}(f,u)_{L_{{p}}}\Big]^{p^*} \frac{du}{u}\right)^{1/p^*}
\quad  \mbox{as } \ \,t\to 0+
\end{equation}
holds for all $f \in L_p(\mathbb R^n)$.
} 

Note that, since LHS(\ref{ulul})$\le$ LHS(\ref{ulul---}), inequality (\ref{ulul}) follows from (\ref{ulul---}). Moreover,
(\ref{ulul---}) provides a sharper bound from below. Indeed, considering $f\in C^\infty$ implies  $\omega_{k}(f,u)_{L_{{r}}}\approx u^k$, $0<u<1$, for any  $1\le r\le \infty$.
Thus,  inequality (\ref{ulul}) even for smooth functions  gives only the rough  estimate $t^{k} \lesssim t^{k-\delta}$ while (\ref{ulul---}) becomes an equivalence.

To prove (\ref{ulul---}), first, we apply Sobolev's embedding 
 $\dot{W}^kL_p\hookrightarrow L_{\bar{p}}^c$ with $1\le p<n/k$ and $1/{\bar{p}}=1/p-k/n$
 (here, as usual, $\dot{W}^kL_p$ is the homogeneous Sobolev space and  $L_{\bar{p}}^c= L_{\bar{p}}/\{constants\}$ is the factor space with the
  norm $\|f\|_{L_{\bar{p}}^c} =\inf_{ c\in \R^1}\|f-c\|_{{\bar{p}}}$).
  See the book \cite[1.77, 1.78]{maly} for the case $k=1.$
For $k>1$, it  follows from the Poincar\'e inequality, namely,
 $$\| f- c\|_{L_{\bar{p}}}\lesssim \|f^\#\|_{L_{\bar{p}}}  \lesssim \|f^\#_{k}\|_{L_p}  \lesssim|f|_{{W^kL_p}}, \quad 1< p< n/k ,$$
  where $ c=\lim_{t\to \infty}f^*(t)$ and $f^\#_{k}$ is the maximal function  given by
  $f^\#_{k}(x)=\sup\limits_{x\in Q}\frac1{|Q|^{1+\frac{k}{n}}}\int_Q |f-P_k f|$, $P_k f$ is a linear projection mapping ${L_1}$ onto the space of polynomials of degree at most $k$, and $f^\#=f^\#_{0}$.
 The first estimate follows from \cite[Corollary 4.3]{bagbykurz} and Hardy type inequalities, the second and third estimates from  \cite[Theorem 9.3, Theorem 5.6, and Corollary 2.2]{devoresharpley}.

 For $p=1$ we obtain by same way
  \[\| f- c\|_{L_{\frac{n}{n-k}, \infty}}\lesssim \|f^\#\|_{L_{\frac{n}{n-k}, \infty}} \lesssim \|f^\#_k\|_{L_{\frac{n}{n-k},\infty}}  \lesssim|f|_{{W^kL_1}}.\]
 By truncated method (\cite[Theorem 7.2.1]{adams}), we can obtain
 \[\| f- c\|_{L_{\frac{n}{n-k}}}\lesssim  |f|_{{W^kL_1}}. \]
By interpolation (see \cite{petunin}),
 \begin{equation}\label{ulul--1}(L_{p}, \dot{W}^kL_p)_{\alpha,{p^*}}
 =\dot{B}^\delta_{p,p*}
 \hookrightarrow L_{p^*}^c =(L_{p}, L_{\bar{p}}^c)_{\alpha,{p^*}}
  \end{equation}
 with $\alpha:=\delta/k$ 
 and $1/{p^*}
  =1/p-\delta/n$. 

On the other hand, since
$L_p\hookrightarrow
L_{p,\infty}=(L_{p^*n/(n+p^*k)}, L_{p^*}^c)_{1-\alpha,{\infty}}$
and
$\dot{W}^kL_{p^*n/(n+p^*k)}
\hookrightarrow L_{p^*}$, we obtain
\begin{align}
\label{ulul--2}
\dot{W}^kL_p\!
&\hookrightarrow\!
\dot{W}^k
(L_{p^*n/(n+p^*k)}, L_{p^*})_{1-\alpha,{\infty}}
\\ \nonumber
&=
(\dot{W}^kL_{p^*n/(n+p^*k)}, \dot{W}^kL_{p^*})_{1-\alpha,{\infty}}
\!\hookrightarrow
\!(L_{p^*}, \dot{W}^kL_{p^*})_{1-\alpha,{\infty}},
\end{align}
where the equality  follows from \cite{peetre} and  \cite{nilson}.

Embeddings (\ref{ulul--1}), (\ref{ulul--2}),
Theorem \ref{emb-kol-holm} (with $\sigma=0$ and $K_0$ instead of $K$), and the known relation
$\omega_k(f,t^{1/k})_{L_{p}}\approx
K_0(f,t; L_{p},{W}^k L_{p})$
  give
  \begin{equation}\label{ulul--2--}
t^{1-\alpha}\sup\limits_{t\le s}\frac{\omega_k(f,s^{1/k})_{L_{p^*}}}
{s^{1-\alpha}}\lesssim \left( \int_{0}^t \Big[ u^{-\alpha}
\omega_{k}(f,u^{1/k})_{L_{{p}}}\Big]^{p^*} \frac{du}{u}\right)^{1/p^*}
\end{equation}
 for $f\in L_p$ and $t >0$ if $0<\alpha<1$ and $1/{p^*}=1/p-\alpha k/n$.
Finally, (\ref{ulul--2--}) and the change of variables  yield (\ref{ulul---}).

\end{example}

Note that in the previous example, we did not use optimal Sobolev embeddings and thus did not obtain the sharp Ul'yanov inequality (\ref{ulya}).
 The optimal embeddings  require to use Lorentz spaces.

\begin{example} 

Let  $n\in \N, \, 1<p< \infty, \,   0 < \delta < n/p$,  $0< \beta,$ and $1 \le {r_0}\le r_1 \le \infty$.
 Take $\theta$ satisfying
  $\max\big\{\frac{\delta}{\delta+\beta}, \frac{p\delta}{n}\big\}<\theta<1$  and set
$$
X=L_{p,{r_0}}, Y=L_{\bar{p},{r_0}}, Z= L_{{p}^*,{r_1}}\quad\mbox{ with}\quad \frac{1}{{p}^*}=\frac{1}{p}-\frac{\delta }{n}, \;\frac{1}{\bar{p}}=\frac{1}{p}-\frac{\delta}{n\theta}.$$ Then, in light of (\ref{embHB--}),
one has $H^{\delta}_{{p},{r_0}} \hookrightarrow L_{p^*,{r_1}}$.
Thus,  the second embedding in (\ref{emb1}) holds with $A^\delta X=H^{\delta}_{p,r_0}$ and $\sigma=\delta$,
where $A$ is the Sobolev integral operator. 


Since
$$H^{\delta+\beta}_{{p},{r_0}} \hookrightarrow  
 L_{\bar{p},{r_0}},
 $$
 which holds by
 \cite[Theorem 1.6 (i)]{see} (note that $\delta+\beta>\frac{\delta}{\theta}=n\big(\frac{1}{p}-\frac{1}{\bar{p}}\big)$),
and
$$
L_{p^*,{r_1}}
 = (L_{{p},{r_0}},L_{\bar{p},{r_0}})_{\theta,{r_1}}\
 \quad\mbox{with}\quad
 \frac{1}{p^*}=\frac{1-\theta}{{p}}+\frac{\theta}{{\bar{p}}}$$
 (see  \cite[Theorem~5.3.1]{belo}),
  we derive that
  (\ref{emb1}) and (\ref{emb2}) hold
with
 $\sigma=\delta$, $\tau=\beta$, and
 the Banach lattice $F_0$ defined as the set of all functions $g\in \mathcal{M}(0,\infty)$
 such that
 $\|g\|_{F_0}=\|u^{-\theta-1/r_1} g(u)\|_{r_1, (0, \infty)}$.

Finally, Theorem \ref{emb-ul-holm} (A) with
$\phi_{0}(t) \approx t^{1/(1-\theta)} =t^{\frac{\beta+\delta}{\beta}}$
implies
\begin{multline*}
\omega_\beta(f,t^{1/\beta})_{L_{p^*,{r_1}}}
\approx K_0(f,t; L_{p^*,{r_1}},H^\beta L_{p^*,{r_1}})
\\
\Big(
\int_0^{t^{\frac{\beta+\delta}{\beta}}}  u^{-\theta {r_1}-1}
\omega_{\beta+\delta}(f,u^{1/{(\beta+\delta)}})_{L_{{p},{r_0}}}^{r_1}
du\Big)^{1/{r_1}}
+
\omega_{\beta+\delta}(f,t^{1/{(\beta+\delta)}})_{L_{{p},{r_0}}}
\Big(\int_{t^{\frac{\beta+\delta}{\beta}}}^\infty u^{-\theta {r_1}-1}du\Big)^{1/{r_1}}
\\
\approx
\Big(
\int_0^{t^{\frac{\beta+\delta}{\beta}}}  u^{-\theta {r_1}-1}
\omega_{\beta+\delta}(f,u^{1/{(\beta+\delta)}})_{L_{{p},{r_0}}}^{r_1}
du\Big)^{1/{r_1}},
\end{multline*}
with the usual modifications for $r_1=\infty.$
The latter is equivalent to the sharp Ul'yanov inequality 
 \quad 
$
\omega_\beta (f,t)_{L_{p^*,{r_1}}} \lesssim \Big(
\int_{0}^{t}
[u^{-\delta}  \, \omega_{\beta +\delta}(f,u)_{L_{{p},{r_0}}}]^{r_1} \frac{du}{u}\Big)^{1/{r_1}}$ {as} $t \to 0+
$
for $f\in L_{{p},{r_0}}$; see (\ref{ulya}). 
  \end{example}




\vskip 0.5cm

\section{The Ul'yanov inequality between  weighted Lorentz spaces}\label{section5}
\subsection{Definitions and preliminaries}
The following definition is motivated by the known result on the equivalence between
 the classical Lorentz space norm and the one involving $f^{**}(t)-f^{*}(t)$, namely,
$$
\|f\|_{L_{{p},r}}\approx \left(\int_0^{\infty} \Big(t^{1/{p}-1/r}\big(f^{**}(t)-f^{*}(t)\big)\Big)^r dt
\right)^{1/r}, \qquad
1<  p, r<\infty,
$$
where $f^{**}(t)=\frac{1}{t}\int_0^t f^*(s)\,ds$ provided that $f^{**}(\infty)=0$, see \cite[Proposition~7.12, p.~384]{besh}.

Let  $X$  be an~r.i. space over  $(\mathbb{R}^n,\lambda_n)$
 and let $w$ be a \textit{weight}, that is, a~nonnegative measurable function on $(0,\infty)$. We define the function gaged cone
$$
S_X(w)(\mathbb{R}^n,\lambda_n):=\{f \in \M(\mathbb{R}^n,\lambda_n):\ f^*(\infty)=0,\ \|f\|_{S_X(w)}:=\|(f^{**}-f^*)w\|_{\bx}<\infty\},
$$
where $\bx$ is a representation space of $X$.

We will also need weighted Lorentz spaces
 defined as follows (cf., e.g., \cite{cgmp}):
If $1\le r<\infty,$ we put
$$\Lambda_r(w)(\mathbb{R}^n,\lambda_n):=\Big\{ f\in \M(\mathbb{R}^n,\lambda_n):\,
\|f\|_{\Lambda_r(w)}:=\Big(\int_{0}^{\infty}\big(f^{\ast}(s)\big)^{r}w(s)\,ds\Big)^{1/r}<\infty
\Big\},$$
$$
\Gamma_r(w)(\mathbb{R}^n,\lambda_n):=\Big\{ f\in \M(\mathbb{R}^n,\lambda_n):
\,  \|f\|_{\Gamma_r(w)}:=\Big(\int_{0}^{\infty}\big(f^{\ast\ast}(s)\big)^{r}w(s)\,ds\Big)^{1/r}<\infty
\Big\},
$$
\begin{multline*}
\quad S_r(w)(\mathbb{R}^n,\lambda_n):=
\Big\{ f\in \M(\mathbb{R}^n,\lambda_n):\ f^*(\infty)=0,
\\
  \|f\|_{S_r(w)}:=\Big(\int_{0}^{\infty}\big(f^{\ast\ast }(s)-f^{\ast}(s)\big)^{r}w(s)\,ds\Big)^{1/r}<\infty\Big\}.
\end{multline*}
We will use the following conditions on weights:
\begin{enumerate}
\item[$\bullet$]\quad
  $w\in B_r$ (i.e., $w$ satisfies the $B_r$ condition) if there is $c>0$ such that
   \[ t^r\int_t^\infty s^{-r} w(s)\,ds\le c\int_0^t w(s)\,ds  \quad \text{for every } \quad t>0;
\]
\item[$\bullet$]\quad
$w\in B_r^*$ (i.e., $w$ satisfies the $B_r^*$ condition)
if there is $c>0$ such that
\[ t^{r}\int_0^t  s^{-r}  w(s)\,ds\le c\int_0^t w(s)\,ds  \quad \text{for every } \quad t>0;
\]
\item[$\bullet$]\quad
$w\in B_\infty^*$ (i.e., $w$ satisfies the $B_\infty^*$ condition)
if there is $c>0$ such that
   \[ \int_0^t \log\frac ts\, \, w(s)\,ds\le c\int_0^t w(s)\,ds  \quad \text{for every } \quad t>0.
\]
\end{enumerate}

In general, $ \Lambda_r(w)(\mathbb{R}^n,\lambda_n)$  and $S_r(w)(\mathbb{R}^n,\lambda_n)$ are
not r.i. spaces, they are
not even linear. On the other hand, ${\Gamma_r(w)(\mathbb{R}^n,\lambda_n)}$ is always an r.i. space for  $1\le r<\infty$ and
in  this case
 the representation space of ${\Gamma_r(w)(\mathbb{R}^n,\lambda_n)}$  is
$\Gamma_r(w)((0,\infty), dt)$.

If $ \Lambda_r(w)(\mathbb{R}^n,\lambda_n)$ is an r.i.~space (e.g., if  $1<r<\infty$ and $w\in B_r$,
see Lemma~\ref{newlemma} below),
 then the representation space of  $\Lambda_r(w)(\mathbb{R}^n,\lambda_n)$ is the space  $\Lambda_r(w)((0,\infty), dt)$.

Similarly, if
$S_r(w)(\mathbb{R}^n,\lambda_n)$ is an r.i.~space (e.g.,  if  $1<r<\infty$ and 
$w\in RB_r$, i.e.   $w(1/t)t^{r-2}\in B_r$;
see \cite[Theorem~3.3]{cgmp}),
 then the representation space of  $S_r(w)(\mathbb{R}^n,\lambda_n)$  is the space $S_r(w) ((0,\infty), dt)$.
Moreover, if $w\in RB_r$, $1<r<\infty$, then
$S_r(w)(\mathbb{R}^n,\lambda_n)$ coincides with ${\Gamma_r(w)(\mathbb{R}^n,\lambda_n)}$.




The \textit{dilation operator} $E_t$, $t\in (0,\infty)$, is defined  on $\Mpl(0,\infty)$ by
$$
  (E_tf)(s):= f\left( ts\right)\quad \mbox{for all \ } s\in (0,\infty).
  $$
Given an r.i.  space $X$ and $t\in (0,\infty)$, the operator $E_t$ is bounded from
$\overline{X}$ to $\overline{X}$ (cf. \cite[p.~148]{besh}). If $h_X$ denotes the {\it dilation function}, i.e.,
$$
h_X(t):=\|E_{1/t}\|_{{\overline X}\rightarrow{\overline X}} \quad  \mbox{for all \ } t\in (0,\infty),
$$
then the \textit{lower and upper Boyd index} of the space $X$ is given by
$$
  \underline{\alpha}_{X}:=\lim_{t\to0+}\frac{\log h_{X}(t)}{\log t}
\qquad\textup{and}\qquad
 \overline{\alpha}_{X}:=\lim_{t\to\infty}\frac{\log h_{X}(t)}{\log t},
$$
respectively. The Boyd indices satisfy (cf. \cite[p.~149]{besh})
$$
  0\leq \underline{\alpha}_{X}\leq \overline{\alpha}_{X}\leq 1.
$$

The \textit{Hardy averaging operator} $P$ and its \textit{dual} $Q$
 are defined on $\Mpl(0,\infty)$, for each $t\in (0,\infty)$, by
$$
(Pf)(t):=\frac{1}{t}\int_0^tf(s)\,ds\ \ \mbox{and} \ \  (Qf)(t):=\int_t^{\infty}\frac{f(s)}{s}\,ds, 
$$
respectively. Recall that (cf.~\cite[p.~150]{besh}) given an r.i.  space $X$, the operator $P$ is bounded on $\bx$
if and only if $\overline{\alpha}_{X}<1$, while the operator $Q$ is bounded on $\bx$ if and only if $0<\underline{\alpha}_{X}$.

We will need the following result, which is partially known but the present formulation  seems to be new. 

\begin{lemma}\label{newlemma} Let $w$ be a weight, $1<r<\infty$, and $X:=\Lambda_r(w)(\mathbb{R}^n, \lambda_n)$.

\noindent
1. The following conditions are equivalent:
\begin{enumerate}
\item[$(a)$]
$w\in B_r$,
\item[$(b)$]
$X$ is an r.i. space,
\item[$(c)$]
the operator $P$ is bounded on $\bx$, 
\item[$(d)$]
$\overline{\alpha}_{X}<1$,
\item[$(e)$]
$X=\Gamma_r(w)(\mathbb{R}^n, \lambda_n)$.
\end{enumerate}

\noindent
2. If $w\in B_r$ and $\eta\in ( 0,1)$, then
the following conditions are equivalent:
\begin{enumerate}
\item[$(a)$]
$w\in B_{q}^*$ with $q=\eta r$,
\item[$(b)$]
the operator
\vskip-0,3cm
$$(Q_\eta f)(t)={t^{-\eta}}\int_t^\infty s^\eta f(s)\frac{ds}{s},\qquad t\in ( 0,\infty),
$$
\vskip-0,1cm
\hskip-0,6cm
 is bounded on $\bx$,
\item[$(c)$]
$\eta<\underline{\alpha}_{X}$.
\end{enumerate}


\noindent
3. If $w\in B_r$, then
the following conditions are equivalent:
\begin{enumerate}
\item[$(a)$]
$w\in B_\infty^*$,
\item[$(b)$]
the operator $Q$
 is bounded on $\bx$, 
\item[$(c)$]
$0<\underline{\alpha}_{X}$.
\end{enumerate}

\end{lemma}
\begin{proof}
Part 1  is known; in more detail, for
$(a) \Leftrightarrow (b)$ see \cite[Theorem~4]{sa},
for $(a) \Leftrightarrow (c)$ see \cite[Theorem~1.7]{armu},
for $(c) \Leftrightarrow (d)$
see  \cite[p. 150]{besh}, and
 $(c) \Leftrightarrow (e)$ is clear.

The proof of part 2 easily follows from 
 the paper \cite[Theorem~3.1]{ne}. 
 The condition
 $w\in B_{\eta r}^*$ is equivalent (cf. \cite[Theorem 3.1]{ne}) to the fact that the operator $Q_\eta$ is bounded in
 $$L_r^\downarrow(w):=
\Big\{f\in \M^+(0,\infty;\downarrow): \|f\|_{L_r^\downarrow(w)}:=\Big(\int_0^\infty |f(t)|^rw(t)\,dt\Big)^{1/r}<\infty\Big\}.$$
It remains to show that the operator $Q_\eta$ is bounded on
 $L_r^\downarrow(w)$ if and only if it is bounded in
 $\Lambda_r(w)$. Part ``if" is clear. To prove the part ``only if",
we first
note that, by Fubini's theorem and  the Hardy--Littlewood rearrangement inequality (see \cite[p.~44]{besh}),
 \begin{equation}\label{444}
\int_0^t (Q_\eta f)(x) \,dx=\frac{1}{1-\eta}\int_0^\infty \min\big(1,\frac{t}{u}\big)^{1-\eta} f(u)\,du
\end{equation}
\vskip-0,2cm
$$\le
 \frac{1}{1-\eta}\int_0^\infty \min\big(1,\frac{t}{u}\big)^{1-\eta} f^*(u)\,du=
 \int_0^t (Q_\eta f^*)(x)\,dx.
$$
Therefore, the fact that $Q_\eta f \in  \M^+(0,\infty;\downarrow),$ the $B_r$ condition, the first part of this lemma,
 inequality (\ref{444}), and the boundedness of  $Q_\eta$ on $L_r^\downarrow(w)$ imply, for any $f\in \Mpl(0,\infty),$ 
  \begin{align*}
 &\Big(\int_{0}^{\infty}\big((Q_\eta f)^*(s)\big)^{r}w(s)\,ds\Big)^{1/r}
 =
 \Big(\int_{0}^{\infty}\big((Q_\eta f)(s)\big)^{r}w(s)\,ds\Big)^{1/r}
 \\&
  \approx
 \Big(\int_{0}^{\infty}\Big(\frac1s\int_0^s (Q_\eta f)(u)\,du\Big)^{r}w(s)\,ds\Big)^{1/r}
 \le
 \Big(\int_{0}^{\infty}\Big(\frac1s\int_0^s (Q_\eta f^*)(u)\,du\Big)^{r}w(s)\,ds\Big)^{1/r}
\\&
  \lesssim
 \Big(\int_{0}^{\infty}\big( (Q_\eta f^*)(s)\big)^{r}w(s)\,ds\Big)^{1/r}
 \lesssim
 \Big(\int_{0}^{\infty}\big(f^*(s)\big)^{r}w(s)\,ds\Big)^{1/r}.
   \end{align*}

The proof of part 3 is similar, one makes use of the fact that
the condition
 $w\in~B_{\infty}^*$ is equivalent to the boundedness of the operator $Q$ on the space
 $L_r^\downarrow(w)$ (cf. \cite[Theorem~3.3]{ne}).
\end{proof}

{ In the rest of this section we work with spaces over $(\mathbb{R}^n, \lambda_n)$ and sometimes we omit the symbol
$(\mathbb{R}^n, \lambda_n)$ from the notation of spaces in question.}

\begin{lemma} \label{lemma5_2} Let $1<r<\infty$, $w\in B_r$,  $\beta \in \mathbb{R}$, and let $v(t):=t^{\beta }$ for all $t\in (0,\infty).$
If $X:= \Lambda_r(w)(\mathbb{R}^n, \lambda_n),$ then
  $$S_X(v)(\mathbb{R}^n, \lambda_n) \hookrightarrow  S_{r}(wv^r)(\mathbb{R}^n, \lambda_n). $$
\end{lemma}
\begin{proof}
Let $\beta \in \mathbb{R}$, $1<r<\infty$, and $f\in \Mpl(\mathbb{R}^n, \lambda_n).$ Since
$\int_t^{2t} s^{\beta-1}\,ds \approx t^\beta$ {for all } $t>0,
$
and since
\begin{equation}\label{104}
\mbox{the function\ \  } t\mapsto t(f^{**}(t)-f^*(t)) \mbox{\ \ is non-decreasing on  } (0,\infty)
\end{equation}
(cf. \cite[Prop.~4.2]{cgo}), on putting
$$g(s):=(f^{**}(s)-f^*(s))s^\beta,\qquad s \in (0,\infty),$$ we obtain that
$$
(f^{**}(t)-f^*(t))t^\beta \lesssim \frac1t\int_t^{2t} g(s)\,ds\le(Pg^*)(t) \quad \mbox{for all } t>0.
$$
Therefore,
$$
\Big(\int_0^\infty (f^{**}(t)-f^*(t))^r(v(t))^rw(t)\,dt\Big)^{1/r}\lesssim \Big(\int_0^\infty ((Pg^*)(t))^rw(t)\,dt\Big)^{1/r}.
$$
Together with the condition  $w\in B_r$ and the first part of Lemma~\ref{newlemma} (recall that in our case
$\overline{X}=\Lambda_r(w)((0,\infty), dt)$), this implies that
$$
\|f\|_{ S_{r}(v^rw)} \lesssim \Big(\int_0^\infty (g^*(t))^r w(t)\,dt\Big)^{1/r} 
=\|g\|_{\overline{X}}=\|f\|_{S_X(v)(\mathbb{R}^n, \lambda_n)},
$$
the required result.
\end{proof}

{ In what follows, 
given $\gamma \ge 0$ and $n\in \mathbb N$, we define the weight
$v_{\gamma, n}$ by
\begin{equation}\label{weight}
v_{\gamma, n}(t):=t\sp{-\frac{\gamma}{n}}\quad \mbox{ for all } t>0.
\end{equation}

\medskip

The next lemma represents a key step in the proof of Proposition~\ref{111} below. It was proved in
 \cite[Theorem~1.1]{gps} for $k=1$, the proof for $k \in \mathbb N$ is analogous.

\begin{lemma}\label{T:kfunct_sx}
If $k,n\in \mathbb N$ and 
 $X$ is an~r.i.~space,
satisfying $ 0< \underline{\alpha}_{X}\leq \overline{\alpha}_{X}< 1$,  then,
for all $t>0$ and $f\in X+S_X(v_{k,n})$,
\begin{align*}
K(f,t;&X,S_X(v_{k,n}))\\
&\approx
\|(f^*(s)-f^*(t))\chi_{(0,t\sp {\frac{n}{k}})}(s)
\|_{\bx}+t\|s^{-\frac kn}(f^{**}(s)-f^*(s))\chi_{(t\sp {\frac{n}{k}},\infty)}(s)\|_{\bx}\nonumber\\
&\approx
\|(f^{**}(s)-f^*(s))\chi_{(0,t\sp {\frac{n}{k}})}
(s)\|_{\bx}+t\|s^{-\frac kn}(f^{**}(s)-f^*(s))\chi_{(t\sp {\frac{n}{k}},\infty)}(s)\|_{\bx}.\nonumber
\end{align*}

\end{lemma}

\begin{proposition} \label{111}  If  $ k,m,n\in {\mathbb N}$,  $1<r<\infty$, and   $w\in B_r\cap B_\infty^*$, then
\[
\left(\Lambda_r(w), S_{\Lambda_r(w)}(v_{k+m,n})\right)_{\frac{m}{k+m},r}=S_{\Lambda_r(wv_{mr,n})}(v_{0,n}).
\]
\end{proposition}

\begin{proof}
Let $X:=\Lambda_r(w)=\Lambda_r(w)(\mathbb{R}^n, \lambda_n)$.
Then the  space $\Lambda_r(w)((0,\infty), dt)$ is the representation space of $X$.
By Lemma~\ref{newlemma},
our assumptions guarantee that $0< \underline{\alpha}_{X}\leq \overline{\alpha}_{X}< 1$.
Therefore, using  Lemma \ref{T:kfunct_sx} (with $k+m$ instead of $k$), we obtain, for all $t>0$ and $f\in X+S_X(v_{k+m,n})$,
\vskip-0,4cm
 \begin{align*}
K(f,t;X,S_X(v_{k+m,n}))\approx&
\left(\int_0^{t^{\frac n{k+m}}}\left(\left[f^{**}(s)-f^*(s)\right]^*\right)^r w(s) ds \right)^{1/r}\\
&\quad+t\left(\int_{t^{\frac n{k+m}}}^\infty
\left(\left[\left(f^{**}(s)-f^*(s)\right)s^{-\frac {k+m}n}\right]^*\right)^r w(s) ds \right)^{1/r}.
\end{align*}
\vskip-0,1cm
If
$$
Y:=\left(\Lambda_r(w), S_{\Lambda_r(w)}(v_{k+m,n})\right)_{\frac{m}{k+m},r},
$$
\vskip-0,1cm
then
  \begin{align*}
  \|f\|_Y=&
 \left(
\int_{0}^\infty
 \left(
t^{-\frac {m}{k+m}}
K(f,t;\Lambda_r(w),S_{\Lambda_r(w)}(v_{k+m,n}))\right)^r \frac{dt}t
\right)^{1/r}
\\
 \approx&
 \left(
\int_{0}^\infty
t^{-\frac {mr}{k+m}}
\int_0^{t^{\frac n{k+m}}}
\left(\left[f^{**}(s)-f^*(s)\right]^*\right)^r w(s) ds
 \frac{dt}t
\right)^{1/r}\nonumber
\\
&\qquad+
 \left(
\int_{0}^\infty
t^{\frac {kr}{k+m}}
\int_{t^{\frac n{k+m}}}^\infty
\left(\left[\left(f^{**}(s)-f^*(s)\right)s^{-\frac {k+m}n}\right]^*\right)^r w(s)ds
 \frac{dt}t
\right)^{1/r}\nonumber
\\
=:&I_1+I_2.\nonumber
%
\end{align*}
Applying Fubini's theorem, we arrive at
\begin{equation}\label{102}
I_1\approx \left(
\int_{0}^\infty \left(\left[f^{**}(t)-f^*(t)\right]^*\right)^r t^{-\frac {mr}{n}}w(t) dt
\right)^{1/r}=\|f\|_{S_{\Lambda_r(wv_{mr,n})}}(v_{0,n})
\end{equation}
and
\begin{equation}\label{103}
I_2\approx \left(
\int_{0}^\infty
t^\frac{kr}{n}\left(\left[
\left(f^{**}(t)-f^*(t)\right)t^{-\frac {k+m}{n}}\right]^*\right)^r w(t) \,dt
\right)^{1/r}.
\end{equation}
Thus, it remains to show that
$
\text{RHS}\eqref{103}\lesssim \text{RHS}\eqref{102}.
$

Let $f\in \Mpl(\mathbb{R}^n, \lambda_n)$ and $g(s):=f^{**}(s)-f^*(s)$ for all $s>0$.
Making use of \eqref{104} and the estimate
\quad$
t^{-\frac{k+m}{n}-1} \approx \int_t^\infty s^{-\frac{k+m}{n}-2} \,ds \quad \mbox{for all } t>0,
$
we obtain that
$$
(f^{**}(t)-f^*(t))t^{-\frac{k+m}{n}} \lesssim \int_t^\infty g(s) s^{-\frac{k+m}{n}-1} \,ds
\quad \mbox{for all } t>0.
$$
Together with the fact that the function $t\mapsto t^{\frac{kr}{n}}$ is non-decreasing on $(0,\infty)$,
this implies that
\begin{align*}
\text{RHS}\eqref{103}&\lesssim \Big(\int_0^\infty t^{\frac{kr}{n}}\Big(\int_t^\infty g(s) s^{-\frac{k+m}{n}-1} \,ds\Big)^{r}
w(t) \,dt\Big)^{1/r}\\
&\lesssim \Big(\int_0^\infty \Big(\int_t^\infty g(s) s^{-\frac{m}{n}-1} \,ds\Big)^{r}
w(t) \,dt\Big)^{1/r}.
\end{align*}
Given $t>0$, we define the non-increasing function $h_t$ by
$$
h_t(s):=\min\{s^{-\frac{m}{n}-1}, t^{-\frac{m}{n}-1}\}\quad \mbox{for all } s>0.
$$
Then
$$
\int_t^\infty g(s) s^{-\frac{m}{n}-1} \,ds\le \int_0^\infty g(s) h_t(s)\,ds \quad \mbox{for all } t>0
$$
and, on applying the Hardy-Littlewood-P\'{o}lya  rearrangement inequality, we arrive at
\begin{align*}
\int_t^\infty g(s) s^{-\frac{m}{n}-1} \,ds&\le \int_0^\infty g^*(s) h_t(s)\,ds \\
&=t^{-\frac{m}{n}-1} \int_0^t g^*(s) \,ds + \int_t^\infty g^*(s) s^{-\frac{m}{n}-1} \,ds\\
&\le(P( g^*(s) s^{-\frac{m}{n}}))(t) +(Q (g^*(s) s^{-\frac{m}{n}}))(t) \quad \mbox{for all } t>0.
\end{align*}
Consequently,
\begin{align*}
\text{RHS}\eqref{103}&\lesssim \Big(\int_0^\infty [(P( g^*(s) s^{-\frac{m}{n}}))(t)]^rw(t)\,dt\Big)^{1/r}+
\Big(\int_0^\infty [(Q( g^*(s) s^{-\frac{m}{n}}))(t)]^rw(t)\,dt\Big)^{1/r}\\
&=:N_1+N_2.
\end{align*}
Making use of the assumption $w\in B_r\cap B_\infty^*$ and Lemma \ref{newlemma}, the fact that
 the function $t\mapsto g^*(t)t^{-\frac{m}{n}}$ is non-increasing on $(0,\infty)$ and the definition of $g$,
we get
\begin{align*}
N_1&\lesssim \Big(\int_0^\infty \big([g^*(t) t^{-\frac{m}{n}}]^*\big)^r w(t)\,dt\Big)^{1/r}\\
&=\Big(\int_0^\infty \big(g^*(t) t^{-\frac{m}{n}}\big)^r w(t)\,dt\Big)^{1/r}\\
&=\Big(\int_0^\infty \big([f^{**}(t)-f^*(t)]^*t^{-\frac{m}{n}}\big)^r w(t)\,dt\Big)^{1/r}\\
&=\text{RHS}\eqref{102}
\end{align*}
and, similarly,
$$
N_2\lesssim\text{RHS}\eqref{102}.
$$
\end{proof}

\begin{lemma}\label{extra}
 If  $m,n\in {\mathbb N}$,  $1<r<\infty$, and   $w\in B_r\cap B_\infty^*$, then
\begin{equation}\label{001}
S_{\Lambda_r(wv_{mr,n})}(v_{0,n})
\hookrightarrow S_{\Lambda_r(w)}(v_{m,n}).
\end{equation}
\end{lemma}
\begin{proof}
Put $X:= \Lambda_r(wv_{mr,n}),$ $Y:=\Lambda_r(w)$. 
Embedding \eqref{001} means that, for all $f \in S_X(v_{0,n})$,
\[
\|(f^{**}-f^*)v_{m,n}\|_{\overline{Y}} \lesssim \|f^{**}-f^*\|_{\overline{X}},
\]
i.e.,
\[
\Big(\int_0^\infty \Big([(f^{**}-f^*)v_{m,n}]^*(t)\Big)^r w(t) \,dt\Big)^{1/r} \lesssim
\Big(\int_0^\infty \Big([f^{**}-f^*]^*(t)\Big)^r w(t)v_{mr,n}(t) \,dt\Big)^{1/r}.
\]
This can be proved quite analogously as the estimate $
\text{RHS}\eqref{103}\lesssim \text{RHS}\eqref{102}.
$
\end{proof}

To prove the needed embeddings
for Sobolev spaces  modelled upon  weighted Lorentz spaces given in Proposition~\ref{1111} below, we  make use of the following lemma,
which is closely related to the results from \cite{MP+} and can be seen as a Sobolev-Gagliardo-Nirenberg type inequality.

\begin{lemma}\label{vsp11}
Suppose that $X(\mathbb R^n)$ is an r.i. space 
such that
$
\frac{k-1}{n}< \underline{\alpha}_{X}$, $k\in \mathbb N$, $k<n$,
and the set of bounded functions is dense in $X$.
Then
\[
\|t^{-\frac{k}{n}}(f^{**}(t)-f^{*}(t))||_{\bx}\lesssim \big\| |D^kf|^*\big\|_{\bx},\qquad f\in W^kX,
\]
where
$$|D^kf|=
\Big(\sum_{|\alpha|=k} |D^\alpha f|^2\Big)^{1/2}.
$$
\end{lemma}

\begin{proof}
First, from  Theorem~2 in \cite{MP+}, 
we have
\begin{equation} \label{milpus}
\|\left(t^{-\frac{k}{n}}(f^{**}(t)-f^{*}(t))\right)^{*}||_{\bx}\lesssim \big\| |D^kf|^*\big\|_{\bx}, \quad f\in C_0^\infty(\R^n).
\end{equation}
Further, we show that the condition $\frac{k-1}{n}< \underline{\alpha}_{X}$ implies $\lim_{t\to 0+} \varphi_X(t)=0$,  where  $\varphi_X$ is the fundamental function of $X$.
Indeed, for  $t\in (0,\frac{1}{2})$, it follows that  $ t^{\frac{1-k}{n}}\lesssim  Q_{\frac{k-1}{n}}\left(\chi_{(0,1)}\right)(t)$ for $k>1$ and  $\log\frac{1}{t}
 \lesssim  Q_{\frac{k-1}{n}}\left(\chi_{(0,1)}\right)(t)$ for $k=1$
and using boundedness of $Q_{\frac{k-1}{n}}$ in $\bx$ (see \cite[Theorem~5.15, p. 150]{besh}), we derive
\begin{align*}\varphi_X(t)&=\|\chi_{(0,t)}\|_{\bx}\lesssim t^{\frac{k-1}{n}} \|Q_{\frac{k-1}{n}}\left(\chi_{(0,1)}\right)\|_{\bx}\lesssim t^{\frac{k-1}{n}} \|\chi_{(0,1)}\|_{\bx} \quad &\text{if}\quad k>1,\\
	\varphi_X(t)&=\|\chi_{(0,t)}\|_{\bx}\lesssim \Big(\log\frac{1}{t}\Big)^{-1} \|Q_{\frac{k-1}{n}}\left(\chi_{(0,1)}\right)\|_{\bx}\lesssim
\Big(\log\frac{1}{t}\Big)^{-1}
 \|\chi_{(0,1)}\|_{\bx} \quad &\text{if} \quad k=1.
\end{align*}
Thus, $\lim_{t\to 0+} \varphi_X(t)=0$.
Using
   \cite[Theorem~5.5, Chapter~2, p. 67]{besh},
  we obtain that $X_a=X_b$
   and   $X_b$ is separable, where
   $X_a$ is the subset of functions $f\in X$ which have absolutely continuous norms and
    $X_b$ is the closure in $X$ of the set of simple functions.
   By our assumption $X=X_b$. Thus $X=X_a=X_b$.
   Then in light of Semenov's theorem (see \cite[Theorem~8, Chapter~II]{KPS}), it  follows that continuous functions are dense in $X_b$.
Further, by standard  density argument,
  one can see that $C_0^\infty(\R^n)$ is dense in $X_b$.
  (For another proof see Remark 3.13 in \cite{ego}. Somewhat similar argument can be found in \cite{shar}.)

  By  Lorentz-Shimogaki result \cite[Theorem~7.4, p.~169]{besh} and
 \cite[Theorem~4.6, p.~61]{besh}, if $\|f_k-f\|_X\to 0$, then $\|f_k^*- f^*\|_{\overline{X}}\to 0$ as $k\to \infty$.
Thus, using a  limiting argument, we may extend the validity of  \eqref{milpus} from functions in $C_0^\infty(\R^n)$
 to  all functions in $W^kX$.

 Since $t(f^{**}(t)-f^{*}(t))$ is an increasing function, we have
\begin{align*}
t^{-\frac{k}{n}}(f^{**}(t)-f^{*}(t))&\lesssim
t(f^{**}(t)-f^{*}(t))\int_t^{2t} x^{-\frac{k}{n}-2}\,dx\\
&\lesssim  Q_{\frac{k-1}{n}}\left(t^{-\frac{k}{n}}(f^{**}(t)-f^{*}(t))\right)(t).
\end{align*}
Taking into account the   condition $\frac{k-1}{n}< \underline{\alpha}_{X}$, the operator $Q_{\frac{k-1}{n}}$ is bounded on  $\bx$ and therefore
\begin{align*}
\|t^{-\frac{k}{n}}(f^{**}(t)-f^{*}(t))||_{\bx}&\lesssim \|Q_{\frac{k-1}{n}}\left(t^{-\frac{k}{n}}(f^{**}(t)-f^{*}(t))\right)||_{\bx}\\
&\lesssim \|\left(t^{-\frac{k}{n}}(f^{**}(t)-f^{*}(t))\right)^{*}||_{\bx}\lesssim \big\| |D^kf|^*\big\|_{\bx}.
\end{align*}
\end{proof}

\begin{proposition} \label{1111} If  $ k,m, n\in {\mathbb N}$, $k+m< n$, $1<r<\infty$, and
$w\in B_r\cap B_{\frac{r(k+m-1)}{n}}^*$,
then
\begin{align}
W^{m}\Lambda_r(w)&\hookrightarrow S_{\Lambda_r(w)}(v_{m, n})
\label{usl1}\\
\intertext{and}
 W^{k+m}\Lambda_r(w)&\hookrightarrow S_{\Lambda_r(w)}(v_{k+m, n}).
\label{usl2}\end{align}
\end{proposition}
\begin{proof}


Set  $X:=\Lambda_r(w)(\mathbb{R}^n, \lambda_n)$. By Lemma \ref{newlemma}, part 2, the assumption
$w\in B_r\cap B_{\frac{r(k+m-1)}{n}}^*$ implies that
$\frac{k+m-1}{n}< \underline{\alpha}_{X}$, which, in turn, gives $\frac{m-1}{n}< \underline{\alpha}_{X}$. Consequently, if
$l \in \{m, k+m\}$, then, by Lemma \ref{vsp11},
\[
 \big\| |D^lf|^*\big\|_{\bx} \gtrsim \|t^{-\frac{l}{n}}(f^{**}(t)-f^{*}(t))||_{\bx},
\]
and embeddings \eqref{usl1} and \eqref{usl2} follow.
\end{proof}

\subsection{The Ul'yanov inequality between weighted Lorentz spaces.}

The next theorem provides an estimate of the $K$-functional $K(f,t; S_r(w), W^kS_r(w))$. 
  Note that in general,
the function gaged cone $S_r(w)$ is not linear (cf., e.g., \cite{cgmp}).
For the definition of the  $K$-functional for the couple $(S_r(w), W^kS_r(w))$ see the discussion
in Section 3. 


\begin{theorem}\label{Sh_Uly_Ineq} If  $ k,m, n\in {\mathbb N}$, $k+m< n$, $1<r<\infty$,
$v_{mr, n}(t)=t\sp{-\frac{mr}{n}}$, and $w\in B_r\cap B_{\frac{r(k+m-1)}{n}}^*$, 
then
\begin{align}\label{004}
K(f,t^k; S_r(wv_{mr,n}),& W^k S_r(wv_{mr,n}))\\
&\lesssim
\left(\int_0^{t}\left(s^{-{m}}K(f,s^{k+m}; \Lambda_r(w), W^{k+m}\Lambda_r(w))\right)^r\frac{ds}{s}\right)^{\frac1r}\notag
\end{align}
for all $\,t>0$ and $f \in \Lambda_r(w)$ 
{\rm (}for which ${\rm RHS}\eqref{004}$ is finite{\rm )}.

\end{theorem}
\begin{proof}
By Lemma \ref{lemma5_2},
\[
Z:=S_{\Lambda_r(w)}(v_{m,n}) \hookrightarrow S_r(wv_{mr,n}),
\]
which implies that
\begin{equation}\label{002}
K(f,t;S_r(wv_{mr,n}), W^k S_r(wv_{mr,n})) \lesssim K(f,t; Z, W^kZ)
\end{equation}
for all  $f \in Z$ and all $t>0$.

To estimate \text{RHS}\eqref{002}, we are going to apply Theorem \ref{emb-ul-holm} (A),
with  $X:=\Lambda_r(w), Y:=S_{\Lambda_r(w)}(v_{k+m,n}),$ the function gaged cone $Z$ mentioned above, with the Sobolev integral operator
as the potential operator $A$, and the Banach lattice $F_0$ defined as the set of all functions $h \in \mathcal{M}(0,\infty)$ such that
$$
\|h\|_{F_0}:=\Big(\int_0^\infty\Big(s^{-\frac{m}{k+m}} |h(s)|\Big)^r \frac{ds}{s}\Big)^{1/r} <\infty.
$$

Note also that the assumption $w\in B_r\cap B_{\frac{r(k+m-1)}{n}}^*$ and Lemma \ref{newlemma} imply that
$w\in B_r\cap B_\infty^*$.

Embeddings \eqref{usl2} and \eqref{usl1} of Proposition \ref{1111} show that assumption \eqref{emb1} of Theorem~\ref{emb-ul-holm} (A)
is satisfied with  $\sigma:=m$ and $\tau:=k$. Using 
Proposition \ref{111}, we arrive at
\[
\left(\Lambda_r(w), S_{\Lambda_r(w)}(v_{k+m,n})\right)_{\frac{m}{k+m},r}=S_{\Lambda_r(wv_{mr,n})}(v_{0,n}).
\]
Since, by Lemma \ref{extra},
\[S_{\Lambda_r(wv_{mr,n})}(v_{0,n})
\hookrightarrow S_{\Lambda_r(w)}(v_{m,n}),
\]
we obtain that
\[
\left(\Lambda_r(w), S_{\Lambda_r(w)}(v_{k+m,n})\right)_{\frac{m}{k+m},r}\hookrightarrow
S_{\Lambda_r(w)}(v_{m,n}), 
\]
which means that assumption \eqref{emb2} of Theorem~\ref{emb-ul-holm} (A) is also satisfied. Consequently, estimate \eqref{emb3}
of Theorem~\ref{emb-ul-holm} (A) implies that
\begin{align}\label{003}
K(f,t; Z, W^kZ)
&\lesssim \left(\int_0^{t^{\frac{k+m}{k}}}\left(s^{-\frac{m}{k+m}}K(f,s; \Lambda_r(w), W^{k+m}\Lambda_r(w))\right)^r\frac{ds}{s}\right)^{\frac1r}
\\
&\qquad \ \
+K(f,t^{\frac{k+m}{k}}; \Lambda_r(w), W^{k+m}\Lambda_r(w))
\left(\int_{t^{\frac{k+m}{k}}}^\infty\left(s^{-\frac{m}{k+m}}\right)^r\frac{ds}{s}\right)^{\frac1r} \notag
\\
&\approx
\left(\int_0^{t^{\frac{k+m}{k}}}\left(s^{-\frac{m}{k+m}}K(f,s; \Lambda_r(w), W^{k+m}\Lambda_r(w))\right)^r\frac{ds}{s}\right)^{\frac1r}.\notag
\end{align}
Combining estimates \eqref{002} and  \eqref{003}, 
we obtain \eqref{004}.
\end{proof}

\begin{remark}\label{remark---} Note that Theorem \ref{Sh_Uly_Ineq} remains true if the $K$-functionals $K$ are replaced by
the $K$-functionals $K_0$ (cf. Remark \ref{remark--}).
\end{remark}

Since  $S_r(w)$ is not a linear space, the calculation of
  the
 $K$-functional \\ $K(f,t; S_r(wv_{mr,n}), W^k S_r(wv_{mr,n}))$
 may cause additional difficulties.
In order to use the previous theorem, we would like to find a Banach function space $Y$ such that $S_r(wv_{mr,n})\hookrightarrow Y$.
The smallest such space $Y$ 
 is  the second associate space $(S_r(wv_{mr,n}))^{\prime \prime}$.

By \cite[Theorem~4.1]{cgmp}, if
\begin{equation} \label{5.3*}
\int_0^\infty t^{-\frac {mr}{n}-r}w(t)dt=\infty,
\end{equation}
 then
\[(S_r(wv_{mr,n}))^{\prime }=\Gamma_{r^\prime}(\overline{w}),
\]
where 
$$
\Gamma_r(\overline{w})=\left\{ f\in \M(\R^n):\quad \|f\|_{\Gamma_r(\overline{w})}:=\Big(\int_{0}^{\infty}\big(f^{\ast\ast}(s)\big)^{r}\overline{w}(s)\,ds\Big)^{1/r}<\infty\right\}
$$
and
\begin{equation}\label{3333--}
\overline{w}(t)=
t^{-\frac {mr}{n}-r}
w(t)
\left(\int_t^\infty
s^{-\frac {mr}{n}-r}
w(s)
ds\right)^{-r^{\prime}}\quad \mbox{for all } t>0.\end{equation}

Now we can  use \cite[Theorem~A]{gk}  to get that
$$
(S_r(wv_{mr,n}))^{\prime \prime }=(\Gamma_{r^\prime}(\overline{w}))^{\prime}=\Gamma_{r}(\nu),
$$
where
\begin{equation}\label{3333}
\nu(t)=\frac{t^{r+r^{\prime}-1}\int_0^t\overline{w}(s)ds \int_t^\infty s^{-r^{\prime}}\overline{w}(s)ds}
{\left(\int_0^t\overline{w}(s)ds + t^{r^{\prime}}\int_t^\infty s^{-r^{\prime}}\overline{w}(s)ds\right)^{r+1}}\quad \mbox{for all } t>0.
\end{equation}
Consequently, $S_r{(wv_{mr,n})}\hookrightarrow \Gamma_{r}(\nu)$. Hence, for all $t>0$ and
$f \in
S_r{(wv_{mr,n})},$
$$
K(f,t; \Gamma_{r}(\nu), W^k \Gamma_{r}(\nu))
\lesssim
K(f,t; S_r(wv_{mr,n}), W^k S_r(wv_{mr,n})).
$$
Thus, using Theorem \ref{Sh_Uly_Ineq} (together with Remark  \ref{remark---}), the facts that
$\Gamma_r(\nu)$ and $\Lambda_r(w)$  are r.i. spaces and that, for all $t>0$,
$$K_0(f,t^k; \Gamma_{r}(\nu), W^k \Gamma_{r}(\nu))\approx\omega_{k}(f,t)_{\Gamma_r(\nu)}$$
and
$$K_0(f,t^{k+m}; \Lambda_r(w), W^{k+m} \Lambda_r(w)) \approx\omega_{k+m}(f,t)_{\Lambda_r(w)}$$
(cf. (\ref{mod-equi})), we arrive at the following
result.

\begin{corollary}\label{--cor}
Let  $ k,m, n\in {\mathbb N}$,
$k+m< n$, and $1<r<\infty$.
Suppose $w\in B_r\cap B_{\frac{r(k+m-1)}{n}}^*$ satisfies \eqref{5.3*} and
$\nu$ is given by {\rm(\ref{3333})};
then
\begin{equation}\label{150}
 \omega_{k}(f,t)_{\Gamma_r(\nu)} \lesssim
\left(\int_0^{t}\left(s^{-m}\omega_{k+m}(f,s)_{\Lambda_r(w)}\right)^r\frac{ds}{s}\right)^{\frac1r}
\end{equation}
for all $\,t>0$ and $f \in \Lambda_r(w)$ {\rm (}for which ${\rm RHS}\eqref{150}$ is finite{\rm )}. Equivalently {\rm (}see Lemma~\ref{newlemma}, Part 1{\rm )},
 we have
$$
 \omega_{k}(f,t)_{\Gamma_r(\nu)} \lesssim
\left(\int_0^{t}\left(s^{-m}\omega_{k+m}(f,s)_{\Gamma_r(w)}
\right)^r\frac{ds}{s}\right)^{\frac1r}.
$$
\end{corollary}

As an important example, we obtain Ul'yanov's inequalities between the Lorentz-Karamata spaces.
To define the Lorentz-Karamata spaces $\, L_{p,r;b}
({\mathbb R}^n),\, 1\le p,r\le \infty,$
we introduce slowly varying functions.

\begin{definition}\label{def}
A measurable function $b :(0,\infty) \to (0,\infty)$ is
said to be {\it slowly varying} on $(0, \infty ),$ notation $b \in
SV(0,\infty )$ if, for each $\, \varepsilon >0,$ there
is a non-decreasing 
function $\, g_\varepsilon $ and a non-increasing 
function $g_{-\varepsilon}$   such that $ \, t^\varepsilon b(t)
\approx g_\varepsilon (t)\, $ and
 $ \, t^{-\varepsilon} b(t) \approx g_{-\varepsilon} (t)$,  for
all $\, t \in (0,\infty ).$
\end{definition}

Clearly, 
$\, \ell^\beta (\ell \circ \ell) ^\gamma \in SV(0,\infty
),$ etc., where $\, \beta,\, \gamma \in {\mathbb R}$ and $\, \ell (t)=
(1+ |\log t |),\, t >0.$

\smallskip
{\bf Convention.}  For the sake of simplicity, in the
following we assume that $\, t^{\pm \varepsilon}b(t)$ are already monotone.

\subsection{The Ul'yanov inequality for Lorentz-Karamata spaces: a  first look.}

We introduce the
{\it Lorentz-Karamata space} $\, L_{p,r;b}({\mathbb R}^n),\, p,r \in
[1,\infty],\, b \in SV(0,\infty),$ as the set of all
measurable functions $\, f $ on $\, {\mathbb R}^n$ such that 
$$\| f\|_{p,r;b}:=
 \Big( \int_{0}^{\infty}[t^{1/p}b(t)f^*(t)]^r\frac{dt}{t}
\Big) ^{1/r}<\infty
$$
(with the usual modification for $r=\infty$). 

 If  $$w(t)=t^{\frac rp-1}b^r(t),\quad 1< p<\infty,\quad 1\le r<\infty, \quad  b\in SV(0,\infty),$$ then  $\Lambda_r(w)=L_{p,r;b}$.
Let $1< r<\infty$.
First we note that the condition $1< p<\infty$  implies that $w\in B_r$. Moreover,
if $p<n/(k+m-1)$, then $w\in  B_{\frac{r(k+m-1)}{n}}^*$.
It is easy to see that the function given by (\ref{3333--}) satisfies
$$
\overline{w}(t)
\approx
t^{\frac{mr'}{n} +\frac{r'}{p'}-1}b^{-r'}(t)\qquad \mbox{for all }t>0.
$$
Furthermore, if $p<n/(k+m-1)$,
then
$$
\int_0^t\overline{w}(s)ds + t^{r^{\prime}}\int_t^\infty s^{-r^{\prime}}\overline{w}(s)ds
\approx t^{\frac{mr'}{n} +\frac{r'}{p'}}b^{-r'}(t)\quad \mbox{for all }t>0,
$$
which, together with  (\ref{3333}), implies that
$$\nu(t)\approx t^{r(\frac1p-\frac{m}n)-1}b^{r}(t)\qquad \mbox{for all }t>0,
$$
and
$$\Gamma_{r}(\nu)=L_{p^*,r;b}\quad\mbox{with} \quad \frac{1}{{p}^*}= \frac{1}{p}-\frac{m}{n}.
$$
 Therefore, by Corollary \ref{--cor}, for all $t>0$ and $f \in L_{p,r;b}$, 
\[
\omega_{k}(f,t)_{L_{p^*,r;b}} \lesssim
\left(\int_0^{t}\left(u^{-m}\omega_{k+m}(f,u)_{L_{p,r;b}}\right)^r\frac{du}{u}\right)^{\frac1r}, \quad \mbox{where } \frac{1}{{p}^*}= \frac{1}{p}-\frac{m}{n}.
\]
Using the estimate $\omega_{k+m}(f,u)_{L_{p,r;b}}\lesssim \omega_{k+m}(f,u)_{L_{p,\overline{r};b}}$ with $\overline{r}\le r$,
we  immediately obtain the following corollary.
\begin{corollary}\label{cor-ry}
If  $ k,m, n\in {\mathbb N}$, $k+m< n$,  $1<
p<n/(k+m-1)$, $1<\overline{r}\le r<\infty$,
 $b\in SV(0,\infty)$, and $1/p^*=1/p-m/n,$ then
\begin{equation}\label{vsp--}
 \omega_{k}(f,t)_{L_{p^*,r; b}} \lesssim
\left(\int_0^{t}\left(u^{-m}\omega_{k+m}(f,u)_{L_{p,\overline{r};b}}\right)^r\frac{du}{u}\right)^{\frac1r} 
\end{equation}
for all $\,t>0$ and $f \in L_{p,\overline{r};b}$ {\rm (}for which ${\rm RHS}\eqref{vsp--}$ is finite{\rm )}.
\end{corollary}
In particular, if $b\equiv 1$, then  \eqref{vsp--} yields the known estimate (\ref{ulya})   for integer   parameters
$k$ and $m$
satisfying  $k+m<n$.
Note that the restriction $\overline{r}\le r$ is natural since  (\ref{vsp--}) does not hold in general for $\overline{r}> r$, see \cite[Theorem~1.1(iii)]{gott}.

In the next section we will investigate inequalities of type (\ref{vsp--})  in more details.

\section{Sharp Ul'yanov inequality between the  Lorentz--Karamata spaces}\label{section6}
In the previous section we
obtained
the Ul'yanov-type inequalities for $K$-functionals and moduli of smoothness between
the general weighted Lorentz spaces, which causes restrictions on the parameters. In particular,
we assumed that $k, m\in \mathbb N$.
On the other hand,
it is clear  that, when dealing with more specific Lorentz spaces, one could get better results,
i.e., sharp Ul'yanov inequalities for a wider range of parameters. 


Our main goal in this section is to establish new sharp Ul'yanov inequalities between the
Lorentz-Karamata spaces 
introduced in the previous subsection.

\medskip
First we mention some simple properties of slowly
varying functions (recall that slowly varying functions have been introduced in Definition \ref{def} at the end of Subsection~5.2). In what follows we write only $SV$ instead of $SV(0,\infty)$.
\begin{lemma} {\rm (cf. \cite [Prop.~2.2]{got})} \label{l2.6+}
Let $b,b_1,b_2 \in SV.$  

\noindent
{\rm (i)} Then $\, b_1 b_2 \in SV,\; b^r \in SV$
and $\, b(t^r) \in SV$ for each $r \in {\mathbb R}.
$ 

\noindent
{\rm (ii)}  If $\varepsilon$ and $\kappa$ are positive numbers,
then
there are positive constants $c_\varepsilon$ and $C_\varepsilon $ \\
\hspace*{.8cm}such that
\[
c_\varepsilon \min \{ \kappa^{-\varepsilon}, \kappa^\varepsilon \}
b(t) \le b(\kappa t) \le C_\varepsilon \max \{ \kappa^\varepsilon,
\kappa^{-\varepsilon} \} b(t) \quad \mbox{for every }\, t>0 .
\]

\noindent
{\rm (iii)} If $\alpha>0$ and $q \in (0,\infty]$, then, for all $t>0,$
\[
\|\tau^{\alpha-1/q} b(\tau)\|_{q,(0,t)} \approx t^\alpha b(t) \quad \mbox{and } \quad
\|\tau^{-\alpha-1/q} b(\tau)\|_{q,(t,\infty)} \approx t^{-\alpha} b(t). \ \
\]
\end{lemma}


\medskip
If $\, b \in SV$, then also $\, b^{-1}:=1/b \in SV$. We will show that these functions have comparable,
sufficiently smooth regularizations:

(a) Given $N\in {\mathbb N}$, following \cite[Lemma~6.3]{guop}, we set
\begin{equation}\label{2.6,5}
a_0(t):= b^{-1}(t)\equiv \frac{1}{b(t)}\, ,\qquad
a_\ell(t) = \frac{1}{t}\int_{0}^{t}
a_{\ell-1}(u)\, du\, ,\qquad  t>0, \ \ \ell \in {\mathbb N},\, 1 \le \ell \le N.
\end{equation}
Then, by direct computation, we obtain, for all $\, t>0$ and $\, \ell\in {\mathbb N},
\;  1 \le \ell \le N,$ that
\begin{equation}\label{2.7}
a_\ell (t) \approx  b^{-1}(t) \quad \mbox{ and } \quad
a_\ell'(t) = - \frac{1}{t}[a_\ell(t) - a_{\ell-1}(t)],
\end{equation}
and hence, for all $\, t>0$ and $\, j, \ell\in {\mathbb N}, \, 1 \le j\le \ell \le N,$
\begin{equation}\label{2.7-++}
|a_\ell^{(j)}(t)|= |t^{-j}\sum_{k=0}^{j}C_{j,k}\, a_{\ell-k}(t)| \lesssim
t^{-j}\, b^{-1}(t) \\
\end{equation}
(with some constants $C_{j,k}$).

(b) Analogously, given $N\in {\mathbb N}$, 
we define
\begin{equation}\label{2.6,6}
c_0(t):= b(t)\, , \qquad c_\ell(t) = t\int_{t}^{\infty}
\frac{c_{\ell-1}(u)}{u^2}\, du\, , \qquad t>0,\ \ \ell \in {\mathbb N},\, 1 \le \ell \le N,
\end{equation}
to obtain, for all $\, t>0$ and $\, \ell\in {\mathbb N},
\;  1 \le \ell \le N,$
\begin{equation}\label{2.8}
c_\ell(t) \approx  b(t)  \quad \mbox{ and } \quad c_\ell'(t)=  \frac{1}{t}
[c_\ell(t) - c_{\ell-1}(t)],
\end{equation}
and hence, for all $\, t>0$ and $\, j, \ell \in {\mathbb N}, \, 1 \le j\le \ell \le N,$
\begin{equation}\label{2.8-++}
|c_\ell^{(j)}(t)| = |t^{-j} \sum_{k=0}^{j} D_{j,k}\, c_{\ell-k}(t)| \lesssim
t^{-j}\, b(t)
\end{equation}
(with some constants $D_{j,k}$).

Now we introduce the subclass $\, SV_{\uparrow}$ of non-decreasing
slowly  varying functions by
{\it
\begin{equation}\label{2.7-}
 SV_{\uparrow}
:= \{ b \in SV: b \mbox{ is non-decreasing,}\; \lim _{t
\to \infty}b(t)=\infty,\; \lim _{t
\to 0_+}b(t)>0\}, 
\end{equation}}
and extend the {\it classical Riesz potential}
\[
I^\sigma f : = k_\sigma *f, \quad \mbox{ where }\quad k_\sigma(x) : =
{\mathcal F}^{-1} [ |\xi|^{-\sigma}](x),\quad 0 < \sigma < n,
\]
to a {\it fractional integration with slowly varying component} %
$\, b^{-1},$ where $b \in SV_\uparrow \, .$ To this end, if the slowly varying function $\, a_N,\; N \in
{\mathbb N},$ is given by (6.1),  set
\[
I_N^{\sigma,b^{-1}} f:= k_{\sigma ,b^{-1};N}*f,\quad \mbox{ where }\quad
k_{\sigma
,b^{-1};N}:= {\mathcal F}^{-1}[|\xi|^{-\sigma}a_N(|\xi|)](x),\; \;  0<
\sigma <n.
\]
When we choose
$\, N >(n+1)/2,$ we can apply  the formula
\[
|{\mathcal F}^{-1}[m(|\xi|^2)](x)| \lesssim \! \int_{0}^{|x|^{-2}}\!\!\!
t^{N-1+n/2} |m^{(N)}(t)|\, dt + |x|^{-N-(n-1)/2} \!\int_{|x|^{-2}}^{\infty}
\!\!\!  t^{N/2 +(n-3)/4} |m^{(N)}(t)|\, dt,
\]
 contained in \cite{tre3}, with $\,
m(t)=t^{-\sigma/2}a_N(\sqrt{t}).$ To this end, observe that
\[
m^{(N)}(t)= \sum_{\ell =0}^{N} t^{-N+\ell -\sigma/2 }\sum_{k=0}^{\ell}
c_{k,\ell,N} \,  a_N^{(k)}(\sqrt{t}) / (t^{\ell/2} t^{(\ell -k)/2})
\]
(with some constants $c_{k,\ell,N}$)
which, by (\ref{2.7-++}), implies that $\, | m^{(N)}(t)| \lesssim
t^{-N-\sigma/2} b^{-1}(\sqrt{t}),$  and hence for
all $\, x \not= 0,$
\[
|k_{\sigma ,b^{-1};N}(x)| \lesssim
\int_{0}^{|x|^{-2}} t^{N-1 +n/2} t^{-N-\sigma/2} b^{-1}(\sqrt{t})\, dt
\qquad \quad \qquad \quad \qquad \quad \qquad \quad \qquad \quad
\qquad \quad
\]
\[
\qquad\qquad\qquad + |x|^{-N-(n-1)/2}
\int_{|x|^{-2}}^{\infty} t^{N/2+ (n-3)/4}t^{-N-\sigma/2}b^{-1}(\sqrt{t})
\, dt \lesssim |x|^{\sigma -n} b^{-1}(|x|^{-1}).
\]
Consequently,
\[
 k_{\sigma ,b^{-1};N}^*(t)
\lesssim k_{\sigma ,b^{-1};N}^{**}(t)
\lesssim  t^{\sigma /n -1} b^{-1}(t^{-1/n})\quad \mbox{for all } t>0\,.
\]
Therefore, the proof of \cite[Theorem~4.6]{edop} can be taken over
to get the following  analog of  a {\it fractional integration} theorem.

\begin{lemma} \label{l2.7}
Let $\, 1<p <\infty,\,  0<\sigma < n/p,\, 1/p^* = 1/p - \sigma /n, \,
1\le r \le s \le \infty ,$  and $ \, B \in SV, b \in SV_{\uparrow}$. If $N\in {\mathbb N}, \  N >(n+1)/2,$ and
\begin{equation}\label{2.8-}
b_n(t):= b^{-1}(t^{-1/n}) \quad \text {for all } \ t>0,
\end{equation}
then
\[
 \| I^{\sigma,b^{-1}}_N f \| _{p^*,s;B} \lesssim \|\,
f\|_{p,r;b_nB} \;
\; \mbox{ for all } \; f \in L_{p,r;b_nB}({\mathbb R}^n)\, .
\]
\end{lemma}

\bigskip
The next lemma deals with a {\it Bernstein inequality} for
slowly varying   derivatives, based on the regularization of $\, b.$
Throughout  this section, given $R>0,$ we put
\[
B_R(0):= \{ \xi \in {\mathbb R}^n: |\xi | \le R \}
\]
and denote by $\chi$  a $C^\infty[0,\infty)$ -- function  such that
\begin{equation}\label{chi}
\chi (u)=1\quad \mbox{if } \quad\, 0 \le u \le 1\quad \mbox{and}
\quad\, \chi (u)=0 \quad \mbox{if } \; u \ge 2.
\end{equation}

\begin{lemma} \label{l2.8}
Let $\, 1<p<\infty,\, 1 \le r \le \infty,$
and $\, g \in L_1({\mathbb R}^n)+L_\infty({\mathbb R}^n)$ with $\, \mbox{{\rm supp}}\, {\widehat g}\subset
B_R(0),\; R> 0.$ If  $\, B \in SV$, $\, b \in
SV_{\uparrow},$ and $N\in  {\mathbb N}, \  N >n/2,$  then
\[
\| {\mathcal F}^{-1}[c_N( |\xi|) {\widehat g}\, ]\,  \|
_{p,r;B} \lesssim  b(R) \,
 \| g\| _{p,r;B} \, \qquad \mbox{for all } R>0, 
\]
where the slowly varying function $c_N$ is given by \eqref{2.6,6}.
\end{lemma}

\begin{proof}
Take $R>0,$  $N\in  {\mathbb N}, \  N >n/2,$  and
define  $$\, m_{R;N}(t):= \chi (t/R)\, c_N(t)/b(R) \quad \text{for all }\ t>0.$$ Then $\, m_{R;N}$
satisfies (cf. \eqref{2.8} and (\ref{2.8-++})) the condition
\begin{equation}\label{2.9}
\sup _{t>0}|m_{R;N}(t)|+ \sup _{\ell \in {\mathbb Z}}
\int_{2^\ell}^{2^{\ell +1}} t^{N-1} |m^{(N)}_{R;N}(t)|\, dt \le C,
\qquad N >n/2,
\end{equation}
which, by e.g. \cite[Theorem~0.2]{bocl}, implies that $\,
m_{R;N}(|\xi|)$ generates a uniformly bounded operator family on
$\, L_p({\mathbb R}^n)\, , \;
1<p<\infty ,$ i.e.,
$\, \| {\mathcal F}^{-1}[m_{R;N}( |\xi|) {\widehat g}\, ]\,  \|_p \lesssim
 \| g\| _{p}\, \;$ if
$1<p<\infty$. 
  Hence, cf. \cite[Corollary~3.15]{ego}, this is also true
for the interpolation space $\, L_{p,r;B}({\mathbb R}^n).$ Since $\, b(R)
m_{R;N}(|\xi|) {\widehat g}= c_N( |\xi|) {\widehat g}$  if $g \in L_1({\mathbb R}^n)+L_\infty({\mathbb R}^n)$
with $\, \mbox{{\rm supp}}\, {\widehat g}\subset
B_R(0),\; $ the assertion
follows.
\end{proof}

A combination of these two lemmas gives the following embedding. 
\begin{lemma}\label{l2.9}
Let $\, 1<p  < \infty,\, 0<\sigma < n/p,\, 1/p^*=1/p - \sigma/n,\,
1\le r\le s \le \infty,\,  \, B \in SV, b \in SV_{\uparrow}$. If
$b_n$ is defined by {\rm (\ref{2.8-})}, then
\[
\| I^{\sigma}g\|_{p^*,s;B} \lesssim b (R)\, \|
g\|_{p,r;b_nB}
\]
for all $R>0$ and for all entire functions $\, g \in L_{p,r;b_nB}({\mathbb R}^n)$ with
$\mbox{{\rm supp}}\, {\widehat g} \subset B_R(0).$
 \end{lemma}

\begin{proof}
Let $N\in {\mathbb N}, \  N >(n+1)/2$ and let the slowly varying functions $a_N, c_N$ be given by  \eqref{2.6,5} and
 \eqref{2.6,6}. Then 
$\, 1 = a_N c_N/ ( a_N c_N)$ on the interval $\, (0,\infty)$. Therefore, supposing that the Fourier symbol
$\, 1/(a_N(|\xi|)c_N(|\xi|))$ generates a bounded operator on
$\, L_p({\mathbb R}^n),\, 1<p<\infty$, then, by Lemma \ref{l2.7} 
 and by Lemma~\ref{l2.8}, we obtain that
\begin{eqnarray*}
\| I^{\sigma}g\|_{p^*,s;B}  & \lesssim  & \| {\mathcal
F}^{-1}[|\xi|^{-\sigma} a_N(|\xi|)c_N(|\xi|) {\widehat g}(\xi)] \,
\|_{p^*,s;B} \\
& \lesssim & \| {\mathcal F}^{-1}[c_N(|\xi|) {\widehat
g}(\xi)] \, \|_{p,r;b_nB} \lesssim b(R) \| g\| _{p,r;b_nB}\,
\end{eqnarray*}
for all $R>0$ and for all entire functions $\, g \in L_{p,r;b_nB}({\mathbb R}^n)$ with
$\mbox{{\rm supp}}\, {\widehat g} \subset B_R(0).$

Thus, by \cite[Theorem~0.2]{bocl}, it remains to show that
$\, 1/(a_N c_N)$ satisfies the condition (\ref{2.9}) (with the function $m_{R;N}$ replaced by $1/(a_N c_N)$).
Introduce the
differential operator $\, D=t(d/dt),$  define $\, D^0$ to be
 the identity operator and $\, D^j=D\, D^{j-1}, \, j \in
{\mathbb N}.$ Now note that  $\, t^N (d/dt)^N$ can
be expressed as a linear combination of $\, D^j,\, 1\le
j \le N,$ that $\, D\, [a_N(t)c_N(t)]= a_{N-1}(t)c_N(t) -a_N(t)c_{N-1}(t)$
and, by induction, that
\begin{equation}\label{2.11}
D^j\, (a_N(t)\, c_N(t)) = \sum_{k =0}^{j} (-1)^{k+1}
{j \choose k} a_{N-k}(t)\, c_{N-j+k}(t)\, ,\quad
\qquad 1 \le j \le N.
\end{equation}
Therefore,
\begin{equation}\label{2.12}\left|D^j
\frac{1}{a_N(t)\, c_N(t)}\right|
\lesssim \sum_{k=1}^{j} \left| \frac{M_{j,k}(t)}{(a_N(t)\, c_N(t))^{k+1}}
\right| \, , \quad 1 \le j \le N, \quad\mbox{for all } t>0,
\end{equation}
where the numerators $\, M_{j,k}(t)$ are appropriate linear combinations of
terms of the type
\[
\prod_{i=1}^{j} \Big\{   D^{\alpha_i^{k,j}}
(a_k(t)\, c_\ell(t))\Big\}^{\beta_i^{k,j}},
\qquad \alpha^{k,j},\beta^{k,j} \in {\mathbb N}_0^j,\; \; \;
\sum_{i=1}^{j}  \alpha_i^{k,j}\beta_i^{k,j}=j.
\]
In view of (\ref{2.6,5}) -- (\ref{2.8-++}), it is clear that the
denominators on the right-hand side of (\ref{2.12}) satisfy
$\, (a_N(t)\, c_N(t))^{k+1}\approx 1\, $ for all $t>0,$ and that, on account of
 (\ref{2.11}), 
 $\,| M_{j,k}(t)|\lesssim 1$ for all $t>0$ if $1\le j \le N$ and $1\le k \le j.$
Therefore,
$\, 1/(a_N c_N)$ satisfies (\ref{2.9}) and the proof is complete.
\end{proof}
The  following  variant of a {\it  Nikol'ski$\breve{\iota}$  inequality} will turn out to be
useful.
\begin{lemma} \label{l2.10}
Let $\, 1<p  < \infty,\, 0<\sigma < n/p,\, 1/p^*=1/p - \sigma/n,\,
1\le r\le s \le \infty,\,  \, B \in SV, b \in SV_{\uparrow}$. If
$b_n$ is defined by {\rm (\ref{2.8-})}, then
\[
\| g\| _{p^*,s;B} \lesssim R^{n(1/p -1/p^*)}b(R)\, \| g\| _{p,r;b_nB}
\]
for all $R>0$ and for all $\, g \in L_{p,r;b_nB}({\mathbb R}^n)$ with
$\mbox{{\rm supp}}\, {\widehat g} \subset B_R(0).$
\end{lemma}
\begin{proof}
Take $\, \chi$ defined by (\ref{chi})  and set $\,
v_R(x):= {\mathcal F}^{-1}[\chi (|\xi|/R)](x), \ x\in {\mathbb R}^n$, $R>0.$ Then, for all
 $x\in {\mathbb R}^n, \ t\in (0,\infty)$ and $R>0,$
\[
|v_R(x)| \lesssim \frac{R^n}{(1+R|x|)^n},\qquad v_R^*(t)
\lesssim \frac{R^n}{(1+R t^{1/n})^n},\qquad
v_R^{**}(t) \lesssim  \min \Big\{ R^n, \frac{1}{t}\Big\}.
\]
By the assumption on the support of the Fourier transform of $\, g,$ we
have  $\, v_R* g=g$. Therefore, by O'Neil's inequality,
\[
g^*(t) =(v_R* g)^*(t) \lesssim t \, v_R^{**}(t)g^{**}(t) +
\int_{t}^{\infty} v_R^*(u) g^*(u)\, du.
\]
Hence, for all $R>0$,\ \ \footnotemark\footnotetext{\ \  We assume that $r,s<\infty.$ If $s=\infty$ or $r=\infty$, then the proof is going along the same lines.}
\begin{eqnarray*}
\| g\|_{p^*,s;B} & \lesssim & \left( \int_{0}^{\infty}\Big[
t^{1/p^*} B (t) \min \Big\{ R^n,\frac{1}{t}\Big\}
\int_{0}^{t} g^*(u)\, du\Big]^s \frac{dt}{t} \right)^{1/s}\\
&& + \, R^n \left( \int_{0}^{\infty}\Big[t^{1/p^*} B (t)
\int_{t}^{\infty} \frac{g^*(u)}{(1+Ru^{1/n})^n} \, du\Big]^s \frac{dt}{t}
\right)^{1/s} =:N_1+N_2.
\end{eqnarray*}
Since $\, t^\varepsilon b_n (t), \, \varepsilon >0,$
is almost non-decreasing and $\, t^{-\varepsilon} b_n (t)$ is
almost non-increasing,  elementary estimates lead to
\begin{eqnarray*}
N_1 & \le & R^n\left( \int_{0}^{R^{-n}} \Big[\, \{t^{1/p^*+1-1/p}
b^{-1}_n(t)\} \,  t^{1/p -1} b_n(t) B(t) 
\int_{0}^{t} g^*(u)\,du \Big]^s \frac{dt}{t} \right)^{1/s} \\
& & \quad + \left(\int_{R^{-n}}^{\infty}\Big[\, \{t^{1/p^* -1/p}
b^{-1}_n(t)\} \, t^{1/p -1} b_n(t) B(t)
\int_{0}^{t} g^*(u)\,du\Big]^s \frac{dt}{t}  \right)^{1/s}\\
& \lesssim & R^{n(1/p-1/p^*)}b(R) \left(
\int_{0}^{\infty}\Big[ \, t^{1/p -1} b_n(t) B(t)
\int_{0}^{t} g^*(u)\,du \Big]^s
\frac{dt}{t}  \right)^{1/s}\quad \mbox{for all } R>0.
\end{eqnarray*}
Now apply a Hardy-type inequality \cite[Lemma~4.1]{gno1} to obtain
\[
N_1 \lesssim R^{n(1/p-1/p^*)}b(R)\,
\| g\| _{p,r;b_n B}\, .
\]
Similarly, handle the term $\, N_2\, ,$ use \cite[Lemma~4.1]{gno1}
 to arrive at
\[
N_2 \lesssim R^n \left( \int_{0}^{\infty}\Big[ t^{1/p^* +1-1/r}
B (t) \frac{g^*(t)}{(1+Rt^{1/n})^n}\Big]^r\, dt \right)^{1/r}
= R^n \left( \int_{0}^{R^{-n}}\ldots + \int_{R^{-n}}^{\infty} \ldots
\right)^{1/r} .
\]
Apply Minkowski's inequality, observe that
\[
 (1+Rt^{1/n})^n \approx
 \left\{ \begin{array}{ll}
 1, \qquad \quad \!  &  0 < t < R^{-n},\\
  R^{n}t, & t \ge  R^{-n},
\end{array} \right.
\]
and use again almost monotonicity properties of $\, t^{\pm \varepsilon}b_n (t)$
to get
 \[N_2 \lesssim R^{n(1/p-1/p^*)} b(R) \,
\| g\| _{p,r;b_nB}\, .
\]
\end{proof}
We will need the Besov-type space $\, B_{(p,r;B),s}^{\,\sigma,b}({\mathbb R}^n)$, modelled
upon  the Lorentz-Karamata space $\, L_{p,r;B}({\mathbb R}^n), 1<p<\infty, 1\le r\le s\le\infty, B\in SV,$  whose
smoothness order $\, \sigma >0$ is perturbed by a  slowly varying
function $\, b \in SV_{\uparrow}.$ To this end, we introduce
the modulus of smoothness of fractional order $\kappa >0$ on the Lorentz-Karamata space $\, L_{p,r;B}({\mathbb R}^n)$
 by  (cf. \eqref{fracdif})
\[
\omega_{\kappa} (f,\delta)_{L_{p,r;B}}:= \sup_{|h|\le \delta }
 \left\|
\Delta_{h}^{\kappa} f(x) \right\|_{L_{p,r;B}({\mathbb R}^n)}
\]
and then we set
\begin{equation}\label{besov}
B_{(p,r;B),s}^{\,\sigma,b}({\mathbb R}^n):=
\Big\{ f \in L_{p,r;B}({\mathbb R}^n): |f|^*_{B_{(p,r;B),s}^{\,\sigma,b}} <\infty \Big\},
\end{equation}
where
\[
|f|^*_{B_{(p,r;B),s}^{\,\sigma,b}}:= \|u^{-\sigma-1/s}b(u^{-1})\,  \omega
_{\kappa+\sigma}(f,u)_{L_{p,r;B}}\|_{s}.
\]
This definition does not depend upon $\, \kappa >0$, which follows from
the  Marchaud inequality (cf. \cite[(1.12)]{trwe1}). 

\medskip
The following lemma is the key result to prove Theorem~\ref{t1+} mentioned below.
\begin{lemma} \label{l2.11}
If $\, 1<p<\infty,\, 0< \sigma <n/p, \,1/p^* =1/p - \sigma /n, \,
1 \le r \le s \le \infty,$ and $\, B \in SV,\, b \in
SV_{\uparrow},$ then
\[
\| f\| _{p^*,s;B} \lesssim
|f|^*_{B_{(p,r; b_nB),s}^{\,\sigma,b}}\; \; \mbox{ for all }\;
f \in B_{(p,r; b_nB),s}^{\,\sigma,b}({\mathbb R}^n).
\]
\end{lemma}

\medskip
The proof follows the same lines as the one of \cite[Lemma~2.6]{gott}.
 Indeed, we use  the  Nikol'ski\v{i} inequality from Lemma~\ref{l2.10}, and the
sequence space $\, \ell_q^\sigma (X),\, X$ a normed space, 
as the
space of $\, X$-valued sequences $\,
(F_j)_{j \in {\mathbb Z}}$ with
\[
\| (F_j)_{j}\| _{\ell_q^\sigma }:= \Big( \sum_{j \in {\mathbb Z}}^{}
[2^{j\sigma} \| b(2^{j})F_j \|_X ]^q \Big)^{1/q} <\infty \, .
\]
Since, by \cite{gus} (see also \cite[Lemma~5.5]{got}), $$\, (L_{p_0^*,s;B},L_{p_1^*,s;B})_{\theta ,q}=
L_{p^*,q;B},\,\ \
1/p^*= (1-\theta)/p_0^* + \theta/p_1^*, \,\ \  0<\theta <1,$$
the rest of the proof of
\cite[Lemma~2.6]{gott} carries over.
\hfill $\Box$
\medskip
\medskip


We will also need the Riesz potential space $H^\lambda_{p,r;\, B}:=H^\lambda L_{p,r;\, B}({\mathbb R}^n)$
modelled upon the Lorentz-Karamata space $\, L_{p,r;B}\, ,\, B\in SV,$ and defined analogously to the space $H^\sigma_{p,r}$
introduced in Subsection~1.2. If $1<p<\infty$, then the estimate
\begin{equation}\label{1.3}
\omega_{\lambda}(f,t)_{L_{p,r;B}} \approx
K_0(f,t^\lambda;L_{p,r;B},H^\lambda_{p,r;\, B}) \quad \mbox{for all }f \in L_{p,r;B} \ \mbox{and }t>0    
\end{equation}
can be verified analogously to estimate (1.13) in \cite [Lemma~1.4]{gott}.

\smallskip
Now we are in a position to prove the sharp Ul'yanov inequality between Lorentz--Karamata spaces.

\begin{theorem}\label{t1+}
Let $\, \kappa >0,\, 1<p<\infty, \, 0< \sigma <n/p, \,
1/p^* = 1/p - \sigma/n,\, 1 \le r\le s \le \infty,$ and $ \,
B \in SV,\, b \in SV_{\uparrow} .$
 If
$b_n$ is defined by {\rm (\ref{2.8-})}, then
\begin{equation}\label{2.13}
\omega_\kappa (f,\delta)_{L_{p^*,s;B}} \lesssim \Big(
\int_{0}^{\delta}
[t^{-\sigma} b(t^{-1}) \, \omega_{\kappa +\sigma}(f,t)_{L_{p,r;b_nB}}]^s
\frac{dt}{t}\Big)^{1/s},\qquad \delta \to 0+,
\end{equation}
for all $\, f\in   B^{\,\sigma,b}_{(p,r; b_nB),s}({\mathbb R}^n).$
\end{theorem}

\smallskip
As an example, recalling that $b$ is non-decreasing, we consider in Theorem \ref{t1+}
$$
b(t)=\begin{cases}
1,   & t\in(0,1]\\
(1+|\ln t|)^\gamma,\ \,\gamma\ge0, &t\in (1,\infty)
\end{cases}
$$
and
$$
B(t)=(1+|\ln t|)^\alpha,\ \,\alpha\in\mathbb{R}, \ \  t\in (0,\infty);
$$
cf. \cite{gott}.


\smallskip
\begin{remark}\label{remark---r} Let all the assumptions of Theorem \ref{t1+} be satisfied.
Note that if $\, f\in L_{p,r; b_nB}({\mathbb R}^n)$ and $\text{RHS}(\ref{2.13})<\infty$
for some $\delta>0,$ then $\, f\in   B^{\,\sigma,b}_{(p,r; b_nB),s}({\mathbb R}^n).$ Indeed, if
$\delta >0,$ then, by Lemma \ref{l2.6+} (iii), for all $\, f\in L_{p,r; b_nB}({\mathbb R}^n)$,
\begin{align*}
 \|u^{-\sigma-1/s}\,b(u^{-1})\,\omega_{\kappa+\sigma}(f,u) _{p,r;b_nB}\|_{s,(\delta,\infty)}
 &\lesssim \|f\|_{p,r;b_nB}\ \|u^{-\sigma-1/s}b(u^{-1})\|_{s,(\delta,\infty)}\\
&\approx \|f\|_{p,r;b_nB}\ \delta^{-\sigma}b(\delta^{-1}).
\end{align*}
Consequently, for all $\, f\in L_{p,r; b_nB}({\mathbb R}^n)$,
\[
|f|^*_{B_{(p,r;b_nB),s}^{\,\sigma,b}}\lesssim \text{RHS}(\ref{2.13})+\|f\|_{p,r;b_nB}\ \delta^{-\sigma}b(\delta^{-1})<\infty,
\]
and the result follows.

Since also
\[
\text{RHS}(\ref{2.13}) \le |f|^*_{B^{\,\sigma,b}_{(p,r;b_nB),s}} \quad \mbox{for all\ \ }
f\in   B^{\,\sigma,b}_{(p,r; b_nB),s}({\mathbb R}^n),
\]
we see that
\[
 B^{\,\sigma,b}_{(p,r; b_nB),s}({\mathbb R}^n) = \{ f \in L_{p,r; b_nB}({\mathbb R}^n):\
\text{RHS}(\ref{2.13}) < \infty\ \, \mbox{for some \ } \delta>0\}.
\]
\end{remark}

\begin{proof} [Proof of Theorem \ref{t1+}] 
  By \eqref{1.3}, for all $\, f\in   B^{\,\sigma,b}_{(p,r; b_nB),s}, \,g \in
H^\kappa _{p^*,s;B}$ and $t>0, $
\begin{eqnarray}\label{112}
\omega_\kappa (f,t)_{L_{p^*,s;B}} & \approx & K_0(f,t^\kappa;
L_{p^*,s;B}, H^\kappa _{p^*,s;B})\\
& \le & \|f-g\|_{p^*,s;B} + t^\kappa
\|(-\Delta)^{\kappa/2}g \|_{p^*,s;B}.\notag
\end{eqnarray}
Take $\, g \in H^{\kappa +\sigma}
_{p,r;b_nB}$ and consider its de la Vall{\'e}e-Pous\-sin
means defined
by
$$ \, g_t:= {\mathcal F}^{-1}[\chi (t|\xi|)] *g,\quad t>0,$$
 where $\, \chi$ is
 the cut-off function from (\ref{chi}). Then 
$\, \mbox{supp} \, {\widehat g_t} \subset B_{2/t}(0).$ Note also that
\begin{equation*}
\|g_t\|_{H^{\kappa +\sigma}_{p,r;b_nB}} \lesssim \|g\|_{H^{\kappa +\sigma}_{p,r;b_nB}} \quad\mbox{for all } t>0
\ \mbox{and } g \in  H^{\kappa +\sigma}_{p,r;b_nB}
\end{equation*}
since $\|{\mathcal F}^{-1} [\chi(t|\xi|)]\|_1 \lesssim 1$ for all $t>0$ by \cite [Corollary~2.3]{tre2}. Thus, using
Lemma \ref{l2.9}, we obtain
\begin{equation}\label{114}
\|(-\Delta)^{\kappa/2}g_t \|_{p^*,s;B}\lesssim  b(1/t)\,
\| (-\Delta)^{(\kappa + \sigma)/2}g_t\|_{p,r;b_nB}
 \quad\mbox{for all } t>0
\ \mbox{and } g \in  H^{\kappa +\sigma}_{p,r;b_nB}.
\end{equation}
Moreover, by Lemma \ref{l2.11},
\begin{equation}\label{115}
\|f-g_t\|_{p^*,s;B}\lesssim |f-g_t|^*_{B_{(p,r;b_nB),s}^{\,\sigma,b}}
 \quad\mbox{for all } t>0
\ \mbox{and } g \in  H^{\kappa +\sigma}_{p,r;b_nB}.
\end{equation}
Combining estimates \eqref{112}-\eqref{115}, we arrive at
$$
\omega_\kappa (f,t)_{L_{p^*,s;B}}  \lesssim  |f-g_t|^*_
{B_{(p,r;b_nB),s}^{\,\sigma,b}} + t^\kappa b(1/t)
 \| (-\Delta)^{(\kappa + \sigma)/2}g_t\|_{p,r;b_nB}
$$
for all $\, f\in   B^{\,\sigma,b}_{(p,r; b_nB),s}, \,g \in
H^{\kappa+\sigma} _{p^*,s;B}$ and $t>0. $

One gets rid of $\, g_t$ estimating $g_t$ by $g$
in a way analogous to the proof of
  \cite [Theorem~1.1~(i)]{gott}. Thus, 
\[
\omega_\kappa (f,t)_{L_{p^*,s;B}}  \lesssim  |f-g|^*_
{B_{(p,r;b_nB),s}^{\,\sigma,b}} + t^\kappa b(1/t)
 \| (-\Delta)^{(\kappa + \sigma)/2}g\|_{p,r;b_nB}.
\]
for all $\, f\in   B^{\,\sigma,b}_{(p,r; b_nB),s}, \,g \in
H^{\kappa+\sigma} _{p^*,s;B}$ and $t>0. $
 Hence, for all $f\in   B^{\,\sigma,b}_{(p,r; b_nB),s}$ and $t>0,$
\begin{equation}\label{2.15}
\omega_\kappa (f,t)_{L_{p^*,s;B}}  \lesssim K_0(f, t^\kappa b(1/t);
 B_{(p,r;b_nB),s}^{\,\sigma,b} ,H_{p,r;b_nB}^{\kappa + \sigma}).
\end{equation}
If we change the variable $\, t^\kappa$ to $\, t^{1-\theta},$ with $\theta =
\sigma/(\kappa+\sigma),$
set $\, b_0(t):=b(t^{-(1-\theta)/\kappa}),\,$ and observe that $\,
B_{(p,r;b_nB),s}^{\,\sigma,b}=(L_{p,r;b_nB},
H_{p,r;b_nB}^{\kappa + \sigma})_{\theta,s;b_0},$ then we can use  the Holmstedt formula
\[
K_0(f,t^{1-\theta}b_0(t); (X,Y)_{\theta,s;b_0},Y) \approx \Big(
\int_{0}^{t}[u^{-\theta} b_0(u) K_0(f,u;X,Y)]^s \frac{du}{u} \Big)^{1/s},
\]
with $X=L_{p,r;b_nB}$ and $Y=H_{p,r;b_nB}^{\kappa + \sigma}$, which
is proved in \cite [Theorem~3.1 c)]{got}. Thus,
\[
K_0(f, t^{1-\theta}b_0(t);
 B_{(p,r;\, b_nB),s}^{\,\sigma,b} ,H_{p,r;\, b_nB}^{\kappa + \sigma})
\approx  \Big( \int_{0}^{t} [u^{-\theta}
b_0(u)K_0(f,u;L_{p,r;\, b_nB},
H_{p,r;\, b_nB}^{\kappa + \sigma})\, ]^s
\frac{du}{u}\Big)^{1/s}
\]
or, after the substitution $\, u=v^{\kappa + \sigma}$ under the integral
sign and after cancelling the change of the \mbox{variable $\, t$,} one
obtains
\[
K_0(f, t^\kappa b(1/t);
 B_{(p,r;\, b_nB),s}^{\,\sigma,b} ,H_{p,r;\,b_nB}^{\kappa + \sigma}) \!\approx\!
\Big(\! \int_{0}^{t}[v^{-\sigma}b(1/v)K_0(f,v^{\kappa + \sigma};L_{p,r;\,b_nB},
H_{p,r;\,b_nB}^{\kappa + \sigma})\, ]^s
\frac{dv}{v}\Big)^{1/s}.
\]
Together with \eqref{2.15} and (\ref{1.3}), this implies the assertion of the theorem.
\end{proof}

\

{\bf Acknowledgments.}
We are very grateful to the anonymous referee for the careful reading of the paper and for the
comments, which helped us to improve the manuscript.

\end{document}